\documentclass[a4paper]{amsart}
\usepackage{amsmath,amsthm,amssymb}
\usepackage{amscd}
\usepackage[dvipdfmx]{graphicx}
\usepackage{mathrsfs}
\usepackage[dvipdfmx, usenames]{color}
\usepackage[all]{xy}

\newtheorem{thm}{Theorem}[section]
\newtheorem{lem}[thm]{Lemma}
\newtheorem{cor}[thm]{Corollary}
\newtheorem{prop}[thm]{Proposition}

\theoremstyle{definition}

\newtheorem{example}[thm]{Example}

\newtheorem{dfn}[thm]{Definition}
\newtheorem{assump}[thm]{Assumption}

\newtheorem{Rem}[thm]{Remark}

\theoremstyle{remark}

\numberwithin{equation}{section}

\def\XXint#1#2#3{{\setbox0=\hbox{$#1{#2#3}{\int}$}
\vcenter{\hbox{$#2#3$}}\kern-.5\wd0}}

\newcommand{\R}{\mathbb{R}}

\newcommand{\calH}{\mathcal{H}}

\newcommand{\calL}{\mathcal{L}}

\newcommand{\calQ}{\mathcal{Q}}

\newcommand{\bbC}{\mathbb{C}}
\newcommand{\bbD}{\mathbb{D}}

\newcommand{\bbR}{\mathbb{R}}

\newcommand{\bbZ}{\mathbb{Z}}

\newcommand{\frakk}{\mathfrak{k}}

\newcommand{\K}{\rm K}

\newcommand{\wtilde}{\widetilde}

\usepackage[abbrev,nobysame]{amsrefs}

\makeatletter
\def\@makefnmark{%
\leavevmode
\raise.9ex\hbox{\check@mathfonts
\fontsize\sf@size\z@\normalfont%
\@thefnmark}%
}
\makeatother

\title[Deformation along orbits and quantization of non-compact manifolds]{Deformation of Dirac operators along orbits and quantization of non-compact Hamiltonian torus manifolds}
\author{H.~Fujita}


\address[H.~Fujita]{Department of Mathematical and Physical Sciences, Japan Women's University, 2-8-1 Mejirodai, Bunkyo-ku Tokyo, 112-8681, Japan}
\email{fujitah@fc.jwu.ac.jp}

\begin{document}
\maketitle
\begin{abstract}
We give a formulation of a deformation of Dirac operator along orbits of a group action 
 on a possibly non-compact manifold to get an equivariant index and a K-homology cycle representing the index. 
 We apply this framework to non-compact Hamiltonian torus manifolds to define geometric quantization from the view point of index theory. 
 We give two applications. 
 The first one is a proof of a [Q,R]=0 type theorem, which can be regarded as a proof of the Vergne conjecture for Abelian case. 
 The other is a Danilov-type formula for toric case in the non-compact setting, 
 which a localization phenomenon of geometric quantization in the non-compact setting. 
 The proofs are based on the localization of index to lattice points. 
\end{abstract}

 \tableofcontents
 %
 
\section{Introduction}
\label{Introduction}
In the present paper we study the following two topics. 
 Firstly, we give a formulation of a deformation of Dirac operator along orbits on a possibly non-compact manifold 
 equipped with a group action to get an equivariant index and a K-homology cycle representing the index. 
 Secondly, we apply this framework to Hamiltonian torus manifolds  
 to define geometric quantization from the viewpoint of index theory. 
 In particular we give proofs of a [Q,R]=0 type theorem and a Danilov-type formula for the toric case in the possibly non-compact setting. 
 The proofs 
 are based on the same perspective, taken in \cite{Fujita-Furuta-Yoshida3, Fujitaorigami} by the author and joint works with Furuta and Yoshida, namely, the {\it localization of index to lattice points}.  
 These results give a simplification and a generalization of  \cite{Fujitaorigami, Fujita-Furuta-Yoshida3}. 
 They also make more clear the relation with a similar construction in \cite{Braverman}.   
 
  Geometric quantization of symplectic manifolds originates from ideas in physics. 
  However, nowadays it is related to several topics in various branches 
  of mathematics. One of them is the index theory of Dirac operator. 
  In fact, in some cases,  the quantization 
  can be regarded as an index of the spin$^c$ Dirac operator associated with a compatible almost complex structure. 
  This approach is called {\it spin$^c$ quantization}. 
  Studying quantization from the viewpoint of index theory, K-theory, K-homology and KK-theory is an active area of research.

  Geometric quantization in the compact setting has been extensively studied. 
  The non-compact case has also been studied to some extent.
  For example, such a generalization is important for quantization of Hamiltonian loop group space in \cite{LoizidesSong}.  
  In addition, the non-compact setting plays an essential role to obtain {\it localization phenomena} in geometric quantization as below. 
  On the other hand, unlike the compact manifold case, the index of Dirac operator on a non-compact or open manifold 
  is not well-defined in a straightforward way. 
  To get the index in a possibly generalized sense,
   it is necessary to take an appropriate boundary condition or to consider additional structure
   such as a fiber bundle structure or a nice group action.   
  
  In \cite{Braverman}, Braverman gave a formulation to define an equivariant index in a non-compact setting. 
  This framework originates in a proof of [Q,R]=0 in \cite{Tian-Zhang} and was applied to a solution of the Vergne conjecture in \cite{Ma-Zhangarxiv}. 
 He used a deformation of the Dirac operator by the Clifford action of the vector field generated by the moment map\footnote{In \cite{Braverman} the formulation is established in a more general category which is not necessarily symplectic. 
 In fact, an equivariant map which is called a {\it taming map} is used. }.
 On the other hand in a series of papers \cite{Fujita-Furuta-Yoshida1, Fujita-Furuta-Yoshida2, Fujita-Furuta-Yoshida3} 
 with Furuta and Yoshida the author developed an index theory on open manifolds 
 using a family of partly defined fiber bundle structures and a deformation of Dirac operator. 
  The deformation in \cite{Fujita-Furuta-Yoshida1, Fujita-Furuta-Yoshida2, Fujita-Furuta-Yoshida3} is given by 
  first-order differential operators, a {\it family of Dirac operators along fibers}, which need not use a group action essentially. 
  We call it FFY's deformation for short.  
  Both Braverman's and FFY's deformation are motivated by Witten's pioneering work \cite{Witten},  and 
  in the equivariant case, these deformations have the same nature, that is, a deformations along the orbits.  
  Both of the resulting indices satisfy the {\it excision formula}, which leads us to the localization of index. 
 Here we summarize the differences between Braverman's and FFY's deformation. 
 
 \begin{itemize}
 \item Braverman's deformation : 
 \begin{enumerate}
\item can be applied to compact group actions (not necessarily Abelian\footnote{Some generalizations to proper actions of 
non-compact Lie groups are established in \cite{Hochs-SongI} for example. }), and 
\item realizes a localization of index to the zero level set of the moment map and fixed points (or critical points of the norm square of the moment map). 
 \end{enumerate}
\item  FFY's deformation :
 \begin{enumerate}
 \item can be applied to torus fibrations (e.g., Lagrangian torus fibrations), and 
 \item realizes a localization of index to the inverse images of the lattice points (or Bohr-Sommerfeld fibers). 
 \end{enumerate}
 \end{itemize}
 As an application of the FFY's second  point above,  a geometric proof of [Q,R]=0 for the torus action case 
 based on the localization of index is obtained in \cite{Fujita-Furuta-Yoshida3}. 
 There is an another application in \cite{Fujitaorigami} which gives a proof of Danilov's formula. 
 Danilov's formula can be regarded as a localization of the geometric quantization of toric manifolds 
 to lattice points in the momentum polytope. 
 The proof in \cite{Fujitaorigami} realizes such a picture of localization faithfully.  
 
 In the present paper we give a framework of a deformation of Dirac operator in a similar manner 
 as in the torus-equivariant setting for FFY's deformation. 
 We use a single differential operator along orbits for the deformation, which satisfies some acyclicity and boundedness condition.  
 We call it an {\it acyclic orbital Dirac-type operator} (Definition~\ref{dfn:exD_K} and Definition~\ref{dfn:K-acyclic orbital system}). 
 Though it is similar to the acyclic compatible system in \cite{Fujita-Furuta-Yoshida1} or \cite{Fujita-Furuta-Yoshida2}, 
 the definition of the acyclic orbital Dirac-type operator is much simpler due to the presence of the global torus action 
 and the isotypic component decomposition of the space of sections. 
 Another difference is that the deformation by an acyclic orbital Dirac-type operator gives an {\it transversally elliptic operator} in the sense of Atiyah~\cite{Atiyah1}. 
 We summarize our first main results : 
 
 \medskip
 
 \noindent
 {\bf Theorem~1. (Theorem \ref{prop:sufficientcondD_K} and Corollary~\ref{cor:D_K->index})
 } 
 {\it 
 Under a suitable technical assumptions we can construct an acyclic orbital Dirac-type operator, 
 which gives an equivariant index valued in the formal completion of the representation ring and a natural {\rm K}-homology cycle representing the index. 
 }
 
 \medskip
 
 The above acyclic orbital Dirac-type operator (Definition~\ref{dfn:exD_K}) is a combination of {\it Kasparov's orbital Dirac operator} \cite{Kasparov} and Braverman's deformation term, 
 which in fact becomes the Braverman type Clifford action shifted by a weight when it is restricted to each isotypic component.  
 The second main result is the following.
 
 \medskip
 
 \noindent
 {\bf Theorem~2. (Theorem~\ref{thm:B=F})}
 {\it Under suitable technical assumptions,  
 the equivariant index defined by the acyclic orbital Dirac-type operator 
 coincides with the equivariant index defined by Braverman's deformation. }
 
 \medskip
 
 \noindent
 As a corollary of Braverman's index theorem in \cite{Braverman}, our equivariant index is also equal to Atiyah's 
 transverse index in \cite{Atiyah1} under the same assumptions.  
  
  Finally we apply the above construction to the setting of non-compact Hamiltonian torus manifolds with possibly non-compact fixed point sets,      
  allowing us to define the spin$^c$ quantization of it as an equivariant index (Definition~\ref{dfn:equivindHam}).
  Our quantization has a localization property to integral lattice points due to its origin. 
  The third main result is the following. 
  
  \medskip
  
  \noindent
  {\bf Theorem~3. (Theorem~\ref{noncpt[Q,R]=0} and Theorem~\ref{thm:noncptDanilov})}
  {\it For the quantization of Hamiltonian torus manifolds defined by an acyclic orbital Dirac-type operator, we have the proofs of the following:  
  \begin{enumerate}
  \item $[Q,R]$=0 theorem for integral regular values of the circle action case, and 
  \item a Danilov-type formula for toric case. 
  \end{enumerate}
  }
  
  \medskip

  The proofs of the above theorems apply also to the compact case, giving  
  simple alternative proofs for \cite{Fujitaorigami, Fujita-Furuta-Yoshida3}\footnote{In fact 
  in \cite{Fujitaorigami} the author showed a Danilov-type formula for {\it toric origami manifolds,} 
  which are a generalization of symplectic toric manifolds. 
  It would be possible to give a proof of a similar formula for non-compact toric origami manifolds  
  by modifying the proof in this paper. }. 
  Since our equivariant index can be identified with Atiyah's transverse index,  
  the proof of the first statement in the above Theorem~3 gives an alternative proof of the Vergne conjecture in \cite{Ma-Zhangarxiv}. 
  In the toric case, the lattice points in the momentum polytope 
  are closely related to the geometric quantization obtained by a {\it real polarization}. 
  There are several results concerning the coincidence between the spin$^c$ (or K\"ahler) quantization and 
  the quantization based on the real quantization 
  from the viewpoint of the index theory.  
  For example see \cite{Andersen, Fujita-Furuta-Yoshida1, Kubota, Yoshida100}.  
  Theorem~\ref{thm:noncptDanilov} can be regarded as such a coincidence in the non-compact setting.

  This paper is organized as follows. 
  In Section~\ref{sec:Example of $D_K$} we construct a $K$-acyclic orbital Dirac-type operator (Definition~\ref{dfn:exD_K} and Theorem~\ref{prop:sufficientcondD_K}) for a complete manifold equipped with an action of a compact torus $K$. 
  This operator arises naturally in the situation of Hamiltonian actions on symplectic manifold. 
  In Section~\ref{sec:Relation with Braverman type deformation} we show that our equivariant index 
  is equal to the equivariant index obtained by Braverman's deformation (Theorem~\ref{thm:B=F}). 
  In Section~\ref{sec:Product fomula} we summarize the product formula in useful two ways 
 (Proposition~\ref{prop:productformula1} and Proposition~\ref{prop:productformula2}). 
  Since the product formula itself can be obtained in the abstract framework of index theory of Fredholm operators 
  we just confirm our set-up and statements. 
  We also present two practical formulas which have key roles in Section~\ref{sec:Riemann-Roch character  of proper Hamiltonian torus space}. 
  In Section~\ref{sec:A vanishing theorem} we show a vanishing formula of index 
  for fixed point subsets (Theorem~\ref{thm:vanish}), which is also important in the construction in Section~\ref{sec:Riemann-Roch character  of proper Hamiltonian torus space}. 
In Section~\ref{sec:Riemann-Roch character  of proper Hamiltonian torus space}, by using the constructions and discussions in 
the previous sections we define quantization of Hamiltonian torus manifolds as an equivariant index (Definition~\ref{dfn:equivindHam}). 
For our quantization we show [Q,R]=0 theorem (Theorem~\ref{noncpt[Q,R]=0}) and 
a Danilov-type formula for toric case (Theorem~\ref{thm:noncptDanilov}). 
The proofs are straightforward from the localization property of our index to lattice points and product formulas.
In Section~\ref{sec:Comments and further discussions} we explain  
some future problems concerning quantization of Hamiltonian loop group spaces and 
a relation between the deformation and KK-product.   
  In Appendix~\ref{sec:Set-up and main assumptions} we give a general machinery to have an 
  equivariant index and a K-homology cycle by using 
  $K$-acyclic orbital Dirac-type operator.  
  We show that a deformation by a $K$-acyclic orbital Dirac-type operator has a compact resolvent 
  on each isotypic component of the space of $L^2$-sections (Corollary~\ref{cor:K-Fredholm}), 
  and hence, it gives an equivariant ($K$-Fredholm) index and 
  a K-homology cycle in a natural way (Definition~\ref{dfn:[M]}). 
  We also show that the resulting Fredholm index is equal to that obtained from a deformation 
  using a large parameter instead of the proper function (Theorem~\ref{thm:FFY=F}). 
  This deformation is closer to the deformation studied in \cite{Fujita-Furuta-Yoshida1, Fujita-Furuta-Yoshida2}. 
  
 \subsection{Notations}
 \label{Notations}
 We fix some notations. 
 
 For  a compact Lie group $K$ let ${\rm Irr}(K)$ be the set of all isomorphism classes of finite dimensional irreducible 
 unitary representations of $K$. 
 We frequently do not distinguish an element $\rho\in {\rm Irr}(K)$ and its corresponding representation space. 
 Each unitary representation $\calH$ of $K$ has the {\it $K$-isotypic component decomposition} 
\[
\calH=\bigoplus_{\rho\in{\rm Irr}(K)}\calH^{(\rho)}, 
\] where each isotypic component $\calH^{(\rho)}$ is defined by 
\[
\calH^{(\rho)}={\rm Hom}_K(\rho, \calH)\otimes\rho. 
\]
We also use the similar notation $A^{(\rho)}$ for the restriction of a $K$-equivariant linear map $A$ to the isotypic component. 
The representation ring of $K$ is denoted by $R(K)$, which is generated by ${\rm Irr}(K)$. 
We denote its formal completion by $R^{-\infty}(K)$, 
namely 
\[
R^{-\infty}(K):={\rm Hom}(R(K),\bbZ). 
\]
  Note that $R(K)$ can be identified with the subgroup consisting of finite support elements in $R^{-\infty}(K)$ 
  by taking the coefficients in each irreducible representation.

 Let $\calH$ be a Hilbert space with inner product $(\cdot, \cdot)$, 
$A$ and $B$ self-adjoint operators on $\calH$ which have common domain. 
We write $A\geq B$ if 
 \[
 (Au, u) \geq (Bu,u)
 \] for all $u\in\calH$ in the domain of $A$. 
 If $\calH$ has a $\bbZ/2$-grading and $A$ is an odd Fredholm operator with the decomposition 
 \[
 A=\begin{pmatrix}
 0 & A^- \\ 
 A^+ & 0
 \end{pmatrix}
 \]according to the grading, then its {\it $\bbZ/2$-graded Fredholm index} is defined as the super dimension of $\ker(A)$; 
 \[
 {\rm index}(A):=\dim(\ker A^+)-\dim(\ker A^-)\in\bbZ. 
 \]
 
 Let $M$ be a Riemannian manifold and $W\to M$ a vector bundle over $M$ equipped with a Hermitian metric 
 $\langle\cdot ,\cdot\rangle_W=\langle\cdot ,\cdot\rangle$. 
 This metric gives rise to an $L^2$-inner product on the space of compactly supported sections $\Gamma_c(W)$ of $W$ 
 which is denoted by $(\cdot, \cdot)_W=(\cdot, \cdot)$. 
 The associated $L^2$-norm and $L^2$-completion are denoted by $\|\cdot\|_W=\|\cdot\|$ and $L^2(W)$ respectively. 

In this paper we mean a generalized Dirac operator by a {\it Dirac(-type) operator}. 
Namely for a vector bundle $W$ over a Riemannian manifold $M$ equipped with a structure of 
a Clifford module bundle over $TM$,  a first-order differential operator $D$ acting on $\Gamma_c(W)$ is called 
a {\it Dirac(-type) operator} if $D$ is a formally self-adjoint operator whose principal symbol is equal to the Clifford action on $W$. 
When $W$ has a $\bbZ/2$-grading we impose that a Dirac operator is an odd operator. 

\subsection{Acknowledgement}
 This work had been done while the author stayed at 
 the Department of Mathematics, University of Toronto and 
 the Department of Mathematics and Statistics, 
 McMaster University. 
 The author would like to thank their hospitality, especially for L.~Jeffrey and  M.~ Harada. 
 He also would like to thank Y.~Loizides for explaining his work and having fruitful discussion 
 about Abelian case. 
The author is partly supported by Grant-in-Aid for Scientific Research (C) 18K03288.
Finally, the author is grateful to the referee for pointing out several mistakes in the preliminary version.


\section{Acyclic orbital Dirac-type operator for torus action}
 \label{sec:Example of $D_K$}
 
 \subsection{Construction of $D_K$}
 \label{subsec:Construction of D_K}

 Let $K$ be a compact torus with Lie algebra ${\frakk}$.  
 We fix an inner product on $\frakk$ and identify ${\frakk}^*={\frakk}$. 
 We often identify ${\rm Irr}(K)$ with $\Lambda^*$, where we put $\Lambda:={\rm ker}({\rm exp}:\frakk\to K)$. 
 Let $M$ be a complete Riemannian manifold and $W$ a $\bbZ/2$-graded Clifford module bundle over $M$. 
 Suppose that $K$ acts on $M$ in an isometric way and the action lifts to $W$ as a unitary action. 
 Take a $K$-invariant Hermitian connection $\nabla$ of $W$.

 For $\xi\in\frakk$ we denote the induced infinitesimal action of $\xi$ on $M$ by $\underline{\xi}^M$. 
 Let $\calL_{\xi}:\Gamma(W)\to \Gamma(W)$ be the induced derivative defined by 
 \[
 \calL_{\xi}s : x\mapsto\left.\frac{d}{dt}\right|_{t=0}{\rm exp}(t\xi)s({\rm exp}(-t\xi)x) 
 \]for $s\in \Gamma(W)$. 
 Let $\mu:M\to {\rm End}(W)\otimes{\frakk}^*$ be the map defined by Kostant's formula; 
 \begin{equation}\label{eq:momentmap}
 \calL_{\xi}-\nabla_{\underline{\xi}^M}=\sqrt{-1}\mu(\xi)=\sqrt{-1}\mu_{\xi} \quad (\xi\in{\frakk}). 
 \end{equation}
 Fix an orthonormal basis $\{\xi_1, \ldots, \xi_n\}$ of $\frakk$. 
 
 \begin{dfn}\label{dfn:exD_K}
 We define the {\it orbital Dirac-type operator} $D_K:\Gamma_c(W)\to \Gamma_c(W)$ by 
\[
D_K:=\sum_{i=1}^nc(\underline{\xi_i}^M)(\calL_{\xi_i}-\sqrt{-1}\mu_{\xi_i}). 
\]
 \end{dfn}

 \begin{lem}\label{lem : lemonD_K}
 The orbital Dirac-type operator $D_K$ satisfies the following conditions. 
 \begin{enumerate}
\item $D_K$ is a first order self-adjoint differential operator which contains only differentials along $K$-orbits. 
\item $D_K$ anti-commutes with the Clifford multiplication of the transverse direction to orbits
 Namely for any $K$-invariant function $h$ on $M$ one has 
\[
D_Kc(dh)+c(dh)D_K=0. 
\]
\item For any Dirac-type operator $D$ acting on $\Gamma(W)$, the anti-commutator 
\[
DD_K+D_KD
\]contains only differentials along $K$-orbits. 
\end{enumerate}
 \end{lem}
 \begin{proof}
 (1) follows from the definition of $D_K$. (2) follows from the anti-commutativity between $c(dh)$ and $c(\underline{\xi_i}^M)$ for any $K$-invariant function $h$. 
By using (2)  one can show (3)  by the computation of the anti-commutator ;  
  \[
 (DD_K+D_KD)h=c(dh)D_K+hDD_K+D_Kc(dh)+hD_KD=h(DD_K+D_KD). 
 \]
 \end{proof}

\begin{Rem}\label{rem:orbital+Braverman}
 The differential term $\displaystyle\sum_{i=1}^nc(\underline{\xi_i^M})\calL_{\xi_i}$ in $D_K$ is the {\it orbital Dirac operator} 
 in the sense of Kasparov \cite{Kasparov}. On the other hand the multiplication term 
 $\displaystyle\sum_{i=1}^nc(\underline{\xi_i^M})\mu_{\xi_i}$ is equal to $c(\underline{\mu})$ for 
 $\underline{\mu}:=\displaystyle\sum_{i=1}^n\underline{\xi_i^M}\mu_{\xi_i}$, which gives the deformation studied by Braverman \cite{Braverman}. 
For each $\rho\in {\rm Irr}(K)$ one has $\calL_{\xi_i}=\sqrt{-1}\rho(\xi_i)$ on each isotypic component $L^2(W_L)^{(\rho)}$ , and hence, 
 \[
D_K^{(\rho)}=\sqrt{-1}\sum_{i=1}^nc(\underline{\xi_i^M})(\rho(\xi_i)-\mu_{\xi_i})=
\sqrt{-1}c(\underline{\rho}-\underline{\mu}), 
 \]
 and 
 \[
 (D_K^{(\rho)})^2=|\underline{\rho}-\underline{\mu}|^2, 
 \]where $\underline{\rho}$ is the infinitesimal action induced by $\rho\in \frakk^*=\frakk$. 
 In other words $D_K$ gives a kind of shift of Braverman's deformation.  
 We investigate the relation between our deformation and Braverman's deformation in the next section. 
 \end{Rem}
 
 For $\rho\in {\rm Irr}(K)$ let $Z_{\rho}:={\rm Zero}(\underline{\rho}-\underline{\mu})$ be the set of points in $M$ at which 
 the vector field $\underline{\rho}-\underline{\mu}$ vanishes. 
 Note that $Z_\rho$ coincides with the set of critical points of $|{\rho}-{\mu}|^2$ in $M$, and it contains $M^K\cup \mu^{-1}(\rho)$.  
 The above description of $D_K$ implies the following. 
 
 \begin{prop}
 For $x\in M$ and $\rho\in {\rm Irr}(K)$ we have 
 \[
 \ker(D_K|_{K\cdot x})^{(\rho)}\neq 0 \Longleftrightarrow x\in Z_{\rho}.   
 \]
\end{prop}
 
 Let $D$ be the Dirac operator acting on $\Gamma(W)$ which is defined by the connection $\nabla$. 
 For each $\rho\in {\rm Irr}(K)$ we put 
 \[
 V_\rho:=M\setminus Z_{\rho}. 
 \]
  Then since $(D_K^{(\rho)}|_{K\cdot x})^2$ is a strictly positive operator on $\Gamma(W|_{K\cdot x})^{(\rho)}$ 
  for any $x\in V_\rho$ there exists a constant $C_{\rho, x}$ such that 
  \[
  |( (DD_{K}+D_{K}D)s, s )_W|\leq C_{\rho,x}( D_{K}^2 s,s)_W
  \]and 
  \[
  |( (D_{K}s, s )_W|\leq C_{\rho,x}( D_{K}^2 s,s)_W
  \]
  hold for any $s\in \Gamma(W_L|_{K\cdot x})^{(\rho)}$. 
  
\begin{thm}\label{prop:sufficientcondD_K}
If the following conditions are satisfied then $(D_K,\{V_\rho\}_{\rho\in{\rm Irr}(K)})$  
is a $K$-acyclic orbital Dirac-type operator on $(M,W)$ in the sense of Definition~\ref{dfn:K-acyclic orbital system}. 
\begin{enumerate}
\item For each $\rho\in{\rm Irr}(K)$, the critical point set $Z_{\rho}$ is compact. 
\item There exists $C>0$ such that 
\[C^{-1}<\displaystyle\sum_{i=1}^n|\underline{\xi_i^M}|<C
\] on the outside of some compact set in $M$. 
\item For each $\rho\in {\rm Irr}(K)$, we have 
\[\sup\{C_{\rho,x} \ | \ x\in V_\rho\}<\infty. 
\]
\item For each $\rho\in{\rm Irr}(K)$,  we have 
\[
 \inf_{x\in V_\rho}\{\kappa \ | \ \kappa \ {\rm is \ the \ minimum \ eigenvalue \ of} \ (D_K|_{K\cdot x})^2 \ {\rm on } \ L^2(W|_{K\cdot x})^{(\rho)} \} >0.
\]
\end{enumerate}
In particular if $M$ has a cylindrical (resp. periodic) end and all the data have translationally invariance (resp. periodicity), 
then the conditions (2),(3) and (4) are satisfied. 
Moreover if there are two such data, then the product of them satisfies these conditions.  
\end{thm}
 
 By using the above $K$-acyclic orbital Dirac-type operator we have a family of deformations of the Dirac operator $D$, 
 \[
 \hat D_\rho=D+f_\rho^4D_K, 
 \]or 
 \[
 D_{\rho,t}=D+t\varphi_\rho^4 D_K \quad (t \gg 0)
 \]as  in Corollary~\ref{cor:K-Fredholm}, Definition~\ref{dfn:[M]} and Corollary~\ref{cor:FFY=F}. 
 As a consequence of Theorem~\ref{prop:sufficientcondD_K}, Corollary~\ref{cor:K-Fredholm}, Definition~\ref{dfn:[M]} and Corollary~\ref{cor:FFY=F} we have the following.   
 
 \begin{cor}\label{cor:D_K->index}
 Under the condition in Theorem~\ref{prop:sufficientcondD_K} the family of deformations $\hat D_\rho$ (or $D_{\rho, t}$) gives a $K$-equivariant index $[\hat D]=[M,W, D_K]\in R^{-\infty}(K)$ and a {\rm K}-homology cycle 
 which represents it. 
 \end{cor}
 

 For later convenience we investigate the case of Hermitian manifold in detail. 
 We assume that the metric on $M$ is induced from a $K$-invariant Hermitian structure $(g, J)$ 
 and the Clifford module bundle $W$ is given by  
 \[
 W=\wedge^{\bullet}T_{\bbC}M\otimes L
 \]
 for a $K$-equivariant Hermitian line bundle with Hermitian connection $(L,\nabla^L)$ over $M$, 
 where $T_{\bbC}M=TM$ is the vector bundle regarded as a complex vector bundle by $J$. 
 This $W$ carries a structure of $\bbZ/2$-graded ${\rm Cl}(TM)$-module bundle with 
 the Clifford multiplication $c:TM\to {\rm End}(W)$ defined by the exterior product and its adjoint. 
 In this case $\mu$ is a map to $\frakk^*$ determined by 
 \[
 \calL_{\xi}^L-\nabla_{\underline{\xi}^M}^L=\sqrt{-1}\mu(\xi)=\sqrt{-1}\mu_{\xi} \quad (\xi\in{\frakk})
 \]and we have 
 \[
\calL_\xi= \calL_{\xi}^M\otimes{\rm id}+{\rm id}\otimes\calL_{\xi}^L. 
 \]
 
 For $x\in M$ let $H^0(K\cdot x; L|_{K\cdot x})$ be the space of global parallel sections on $(L, \nabla^L)|_{K\cdot x}$, 
 which is a vector space of dimension  at most one.
 Suppose that $H^0(K\cdot x; L|_{K\cdot x})\neq 0$ and $s$ is its non-trivial element, then we have 
 \[
 0=\nabla^L_{\xi}s=(\calL_{\xi}^L-\sqrt{-1}\mu_{\xi})s
 \]for all $\xi\in \frakk$. 
 This equation implies that $\mu_\xi(x)$ is an integer for all $\xi$, 
 and hence, we have 
 the following.  
 
 \begin{prop}\label{prop:BS->Z}
 If $H^0(K\cdot x; L|_{K\cdot x})\neq 0$ for $x\in M$, then we have 
 $\rho:=\mu(x)\in\Lambda^*$ and $H^0(K\cdot x; L|_{K\cdot x})=\bbC_{(\rho)}$, 
 where $\bbC_{(\rho)}$ is the 1-dimensional representation of $K$ whose weight is given by $\rho$. 
 In particular if $H^0(K\cdot x; L|_{K\cdot x})\neq 0$, then we have $K\cdot x\subset Z_\rho$ for $\rho=\mu(x)\in\Lambda^*$. 
 \end{prop}
 
 \begin{Rem}
 If $M=(M,\omega)$ is a symplectic manifold whose dimension is twice of the dimension of $K$,  
 the $K$-action is an effective Hamiltonian torus action 
 and $(L,\nabla^L)$ is a prequantizing line bundle, i.e., the curvature form of $\nabla^L$ is equal to $-\sqrt{-1}\omega$, 
 then the condition $H^0(K\cdot x; L|_{K\cdot x}) \neq 0$ is equivalent to the {\it Bohr-Sommerfeld condition for the orbit $K\cdot x$}, 
 which is essential in the geometric quantization by the {\it real polarization}. 
  
 \end{Rem}

\subsection{Non-complete case and localization formula}
\label{subsec:Non-complete case and localization formula} 
As we will mention in the end of Appendix~\ref{subsec:Relation with Fujita-Furuta-Yoshida type deformation} 
the index associated with the $K$-acyclic orbital Dirac-type operator can be defined for non-complete situation. 
For instance suppose that the first condition\footnote{In Section~\ref{sec:Riemann-Roch character  of proper Hamiltonian torus space} we handle with the 
non-compact fixed point set case using the vanishing of index (Theorem~\ref{thm:vanish}).  } 
 in Theorem~\ref{prop:sufficientcondD_K} is satisfied. 
We take a $K$-invariant compact submanifold $X_\rho$ with boundary as a neighborhood of $Z_{\rho}$ 
and attach a cylinder $\partial X_\rho\times [0,\infty)$ to $\partial X_\rho$ so that we have a $K$-invariant complete Riemannian manifold 
$\tilde X_\rho$ with $K$-invariant cylindrical end. 
Let $\tilde\mu$, $\tilde W$ and $\tilde D_K$ be the extensions of $\mu$, $W$ and $D_K$ on $\tilde X_\rho$ such that 
they have translational invariance and 
$\ker(\tilde D_K^{(\rho)}|_{K\cdot x})=\ker((\tilde D_K^{(\rho)}|_{K\cdot x})^2)=\ker(|\underline{\rho}-\underline{\mu}|^2)=0$ 
for any $x\in \partial X_\rho\times (0,\infty)$. 
These data define a Fredholm operator on $L^2(\tilde W)^{(\rho)}$ as in Corollary~\ref{cor:K-Fredholm}. 
Though we agree that it is a little bit strange notation\footnote{The excision formula guarantees that this index 
defined on a neighborhood of $Z_\rho$ does not depend on a choice of the neighborhood. }, we denote this index by 
\begin{equation}\label{eq:fixedpoint}
[Z_\rho]\in \bbZ. 
\end{equation}
 We decompose 
 \[
 Z_\rho=\mu^{-1}(\rho)\cup\left(\bigcup_{\alpha} Z_{\rho,\alpha}\right)
 \]into the disjoint union of the connected components, where $Z_{\rho,\alpha}$ is a 
 connected component other than $\mu^{-1}(\rho)$. 
 This description enable us to get more refined decomposition of (\ref{eq:fixedpoint})  into the summation of 
 local contributions from each component, which we denote by 
 \[
 [Z_\rho]=[\mu^{-1}(\rho)]+\sum_{\alpha} [Z_{\rho,\alpha}]. 
 \]
The excision formula implies the following localization formula. 

\begin{thm}\label{thm:localizationformula}
If the conditions in Theorem~\ref{prop:sufficientcondD_K} are satisfied, then 
the index $[\hat D]=[M, W, D_K]\in R^{-\infty}(K)$ defined by the $K$-acyclic orbital Dirac-type operator $D_K$ satisfies
\[
[M,W,D_K](\rho)=[Z_\rho]=[\mu^{-1}(\rho)]+\sum_{\alpha} [Z_{\rho,\alpha}]
\]for each $\rho\in {\rm Irr}(K)$. 
\end{thm}

In Section~\ref{sec:Riemann-Roch character  of proper Hamiltonian torus space} we discuss 
the case of Hamiltonian circle action and symplectic toric case. 
In these cases one has the vanishing $[Z_{\rho,\alpha}]=0$, and hence, 
we realize the localization of the index $[M,W,D_K]$ into the lattice points ${\rm Irr}(K)=\Lambda^*$.

 \subsection{Relation with Braverman's deformation}
 \label{sec:Relation with Braverman type deformation}
In \cite{Braverman} Braverman studied 
 a Witten-type deformation of the Dirac operator and its equivariant index on non-compact $K$-manifold.
In a symplectic geometric setting Braverman's deformation is given by the Clifford multiplication of the Hamiltonian vector field of the norm square of the moment map. In particular in the setting in Section~\ref{subsec:Construction of D_K}  (not necessarily $K$ is a torus) we can 
consider the Braverman's deformation as 
\[
D_{\mu}:=D-h\sqrt{-1}c(\underline{\mu}),  
\]where $h:M\to\bbR$ is a $K$-invariant function called an ${\it admissible \ function}$ which satisfies a suitable growth condition. 
Braverman showed several fundamental properties of $D_{\mu}$. 
In particular he showed that $D_{\mu}$ is a $K$-Fredholm operator and the resulting index in $R^{-\infty}(K)$ is 
independent of a choice of the admissible function. 
Moreover the index is equal to Atiyah's transverse index. 
After that his equivariant index has been applied in several directions, for instance, a solution to Vergne's conjecture by Ma-Zhang  \cite{Ma-Zhangarxiv}. 

In this section we consider the same set-up in Section~\ref{subsec:Construction of D_K} 
and assume the followings to make the situation simple. 
\begin{assump}\label{assump:B=F}
 We assume that the conditions in Theorem~\ref{prop:sufficientcondD_K} are satisfied together with the 
cylindrical end condition\footnote{The cylindrical end condition is used to have a uniform estimate on the end. 
It is possible to put weaker assumptions to have the uniform estimate. 
For example we can handle with products of manifolds with cylindrical end.} and ;
 \begin{itemize}
\item The moment map $\mu:M\to {\rm End}(W)\otimes\frakk^*$ defined by Kostant's formula (\ref{eq:momentmap}) 
is proper in the sense that each inverse image of a compact subset of $\frakk$ by $\mu$ is compact. 
\item The differential of the function $|\mu|:M\to \bbR$ is $L^{\infty}$-bounded on the outside of the compact subset $\mu^{-1}(0)$. 
 \end{itemize}
 Note that the second condition is satisfied for the symplectic setting and the genuine moment map $\mu$ by taking $J$ as an 
 $\omega$-compatible almost complex structure. 
\end{assump}
 
 We show the following. 
 
 \begin{thm}\label{thm:B=F}
 Under Assumption~\ref{assump:B=F} we have 
 \[
 {\rm index}_K(D_{\mu})=[\hat D]\in R^{-\infty}(K). 
 \]
 \end{thm}
 
 \begin{Rem}
 As it is noted in \cite[Example~5.2]{FujitaS1} the above equality does not hold in general without 
 properness of $\mu$ or completeness of $M$.  
 \end{Rem}
 
 As a corollary of Braverman's index theorem (\cite[Theorem~5.5]{Braverman}) we also have the following. 
 
 \begin{cor}\label{cor:F=Transverse}
Under Assumption~\ref{assump:B=F} $[\hat D]=[M, W,D_K]\in R^{-\infty}(K)$ is equal to the 
transverse index in the sense of Atiyah \cite{Atiyah1}.  
\end{cor} 
 
 We first note that under Assumption~\ref{assump:B=F} 
 we can take $f$ as in Appendix~\ref{sec:Set-up and main assumptions} so that $f=|\mu|$ on the outside of a 
 compact neighborhood of the compact subset $\mu^{-1}(0)$. 
 Moreover we can take an admissible function $h$ to be $f_{\rho}^4=\varphi_{\rho}^4f^4$ for each $\rho\in{\rm Irr}(K)$, where $\varphi_\rho$ is the cut-off function for $V_{\rho}=M\setminus Z_{\rho}$ as in (\ref{eq:cut-off fct}). 

To show  Theorem~\ref{thm:B=F} we fix $\rho\in{\rm Irr}(K)$ and consider the following 1-parameter family in the setting in Section~\ref{subsec:Construction of D_K} : 
\[
{\bbD}_\epsilon:=D+\epsilon f_{\rho}^4D_K-(1-\epsilon)\sqrt{-1}f_\rho^4c(\underline{\mu}) 
\]for $\epsilon\in [0,1]$.
We show that for each $\rho$ an unbounded operator $\bbD_\epsilon^{(\rho)}$ on $L^2(W)^{(\rho)}$ 
gives a norm-continuous family of the bounded transformations such as 
$\displaystyle\frac{\bbD_\epsilon}{\sqrt{1+\bbD_\epsilon^2}}$, and hence, 
the equality 
\[
{\rm index}_K(D_\mu)(\rho)={\rm index}((\bbD_\epsilon)^{(\rho)})={\rm index}(\hat D_\rho)=
[\hat D](\rho)
\]holds. 
We use the following criteria. 

\begin{lem}[Proposition~1.6 in \cite{Nicolaescu}]\label{lem:Nico} 
Let $A_0$ and $A$ be unbounded self-adjoint
operators on a Hilbert space such that ${\rm dom}(A_0) \cap {\rm dom}(A)$ is dense. 
Suppose that the family of operators
$A_\epsilon = A_0 + \epsilon A \ (\epsilon  \geq 0)$ is essentially self-adjoint and 
for each $\epsilon\geq 0$ the following conditions hold:
\begin{enumerate}
\item $A_\epsilon$ has a gap in its spectrum. 
\item ${\rm dom}(A_\epsilon)\subset {\rm dom}(A)$
\item There exists constants $C, C'>0$ such that $C' A^2\leq A_\epsilon^2+C$. 
\end{enumerate} 
Then the family of bounded transforms $\epsilon\mapsto \frac{A_\epsilon}{\sqrt{1+A_\epsilon^2}}$ is norm-continuous. 
\end{lem}

As in \cite[Remark~4.10]{Loizides-Song2} it suffices to show the third condition in Lemma~\ref{lem:Nico} in our situation.

Hereafter we mainly consider the isotypic component of operators.  
Even if so we often omit the superscript $(\cdot)^{(\rho)}$ of the isotypic component for simplicity and 
use the notation as $D:L^2(W)^{(\rho)}\to L^2(W)^{(\rho)}$ and so on. 
As we noted in Remark~\ref{rem:orbital+Braverman} 
one can write as $\displaystyle D_K=\sqrt{-1}c(\underline{\rho}-\underline{\mu})$ on $L^2(W)^{(\rho)}$, and hence,  we have 
\begin{eqnarray*}
\bbD_\epsilon&=&D+ f_\rho^4\sqrt{-1}\left(\epsilon c(\underline{\rho}-\underline{\mu})-(1-\epsilon)c(\underline{\mu})\right) \\ 
&=& D+f_{\rho}^4\sqrt{-1}c(\epsilon\underline{\rho}-\underline{\mu}) \\
&=&D-f_\rho^4\sqrt{-1}c(\underline{\mu})+\epsilon f_\rho^4\sqrt{-1}c(\underline{\rho}) . 
\end{eqnarray*}

Then the third condition in Lemma~\ref{lem:Nico} is equivalent to 
\[
C'\left(f_\rho^4\sqrt{-1}c(\underline{\rho})\right)^2
\leq (\bbD_\epsilon)^2+C. 
\]for some constants $C, C'>0$. 
Since 
\[
\left(\sqrt{-1}c(\underline{\rho})\right)^2=
\left|\underline{\rho}\right|^2\leq \sum_{i}|\rho(\xi_i)\underline{\xi_i^M}|^2
\]by using an orthonormal basis $\{\xi_1,\ldots, \xi_n\}$ of $\frakk$, 
and our boundedness condition on $|\underline{\xi_i^M|}$ it suffices to show the following.  

\begin{lem}\label{lem:Nico2}
There exist constants $C, C'>0$ such that 
\[
C'f_{\rho}^8\leq (\bbD_\epsilon)^2+C 
\]
holds for all $\epsilon\in [0,1]$. 
\end{lem}

\begin{proof}On $L^2(W)^{(\rho)}$ we have 
\begin{eqnarray*}
(\bbD_\epsilon)^2&=&D^2+\sqrt{-1}\left(Df_\rho^4c(\epsilon\underline{\rho}-\underline{\mu})+f_\rho^4c(\epsilon\underline{\rho}-\underline{\mu})D\right)+
f_\rho^8|\epsilon\underline{\rho}-\underline{\mu}|^2 \\
&=&D^2+\sqrt{-1}\left(4f_\rho^3 c(df_\rho)c(\epsilon\underline{\rho}-\underline{\mu})+
f_{\rho}^2(Dc(\epsilon\underline{\rho}-\underline{\mu})+c(\epsilon\underline{\rho}-\underline{\mu})D)f_\rho^2\right)
+f_\rho^8|\epsilon\underline{\rho}-\underline{\mu}|^2. 
\end{eqnarray*}
On the other hand there exist constants $C_1>0$ and $C_2>0$ such that 
\begin{eqnarray*}
|c(df_\rho)c(\epsilon\underline{\rho}-\underline{\mu})|\leq 
\|df_\rho\||\epsilon\underline{\rho}-\underline{\mu}|
\leq C_1 |\epsilon\rho-\mu|
\end{eqnarray*}and 
\[
|f_{\rho}^2(Dc(\epsilon\underline{\rho}-\underline{\mu})+c(\epsilon\underline{\rho}-\underline{\mu})D)f_\rho^2|
\leq C_2|\epsilon\rho-\mu|f_{\rho}^4D_K^2=C_2f_\rho^4|\epsilon\rho-\mu|^3,  
\]
where we get the inequality in a similar way as the proof of Proposition~\ref{prop:compactresolvent} 
and we use the assumption on cylindrical end so that we can take $C_2$ uniformly.  
So we have 
\[
(\bbD_\epsilon)^2\geq  -4C_1 f_{\rho}^3|\epsilon\rho-\mu|-C_2f_\rho^4|\epsilon\rho-\mu|^3+
f_\rho^8|\epsilon\underline{\rho}-\underline{\mu}|^2. 
\]
On the other hand since $\mu$ is proper and $M^K$ is compact  
$|\epsilon\underline{\rho}-\underline{\mu}|$ is uniformly positive on the outside of a compact subset, 
and hence, there exists $C'>0$ such that 
\[
f_{\rho}^8|\epsilon\underline{\rho}-\underline{\mu}|+1>2C'f_{\rho}^8. 
\]
Since $f_\rho=|\mu|$ on the outside of a compact subset there exists $C>0$ independent from $\epsilon\in [0,1]$ such that 
\[
-4C_1 f_{\rho}^3|\epsilon\rho-\mu|-C_2f_\rho^4|\epsilon\rho-\mu|^3-1+C'f_{\rho}^8>-C. 
\]
Finally we have 
\[
(\bbD_\epsilon)^2>(1-C'f_{\rho}^8-C)+(2C'f_{\rho}^8-1)=C'f_{\rho}^8-C
\]
and hence, 
$(\bbD_\epsilon)^2+C>C'f_\rho^8$. 
\end{proof}

\section{Product fomula}
\label{sec:Product fomula}
For later convenience we summarize the product formula for our index and some useful formulas derived from it. 
Instead of giving full general setting we explain typical two situations which will be used in the subsequent sections.  
We follow the basic formulation of the product formula of indices as in \cite{Atiyah-Singer}, and 
we give a formulation to adapt that in \cite[Section~3.3]{Fujita-Furuta-Yoshida2}.
Though we use terminologies in Appendix~\ref{sec:Set-up and main assumptions}, 
the main applications are the acyclic orbital Dirac-type operators constructed in Section~\ref{sec:Example of $D_K$} and Theorem~\ref{prop:sufficientcondD_K}.

\subsection{Direct product}
\label{subsec:Direct product}
For $i=0,1$ let $K_i$ be a torus. 
Let $M_i$ be a complete Riemannian manifold and $W_i\to M_i$ a $\bbZ/2$-graded Clifford module bundle on which 
$K_i$ acts in an isometric way. 
Suppose that there exists a $K_i$-acyclic orbital Dirac-type operator $(D_{K_i}, \{V_{i,\rho_i}\}_{\rho_i\in {\rm Irr}(K_i)})$ 
on $(M_i, W_i)$. 
Put $M:=M_0\times M_1$ and define a Clifford module bundle $W$ over $M$ by the outer tensor product 
\[
W:=W_0\boxtimes W_1 
\] for the projections onto the first and second factor of $M$. 
For $\rho=(\rho_0, \rho_1)\in {\rm Irr}(K_0)\times {\rm Irr}(K_1)$ we define $V_{\rho}$ by 
\[
V_{\rho}:={V_{0,\rho_0}}\times V_{1,\rho_1} 
\]whose complement in $M$ is compact. 
Let $D_K:\Gamma(W)\to \Gamma(W)$ be an operator defined by 
\[
 D_K:= D_{K_0}\otimes {\rm id}+\varepsilon_{W_0}\otimes D_{K_1}=D_{K_0}+\varepsilon_{W_0}D_{K_1}, 
\]where $\varepsilon_{W_0}: W_0\to  W_0$ is the grading operator on $W_0$. 
Since $D_{K_0}(\varepsilon_{W_0}D_{K_1})+(\varepsilon_{W_0}D_{K_1})D_{K_0}=0$ 
one has the following. 
\begin{lem}
$(D_K, \{V_{\rho}\}_{\rho\in {\rm Irr}(K)})$ is a $K$-acyclic orbital Dirac-type operator on $(M,W)$. 
\end{lem}
Dirac operators $D_i$ on $W_i$ give rise the Dirac operator $D$ on $W$; 
\[
 D:= D_{0}\otimes {\rm id}+\varepsilon_{ W_0}\otimes D_{1}= D_{0}+\varepsilon_{ W_0}D_{1}. 
\]
For each $\rho_i\in {\rm Irr}(K_i)$ we take a $K_i$-invariant cut-off function $\varphi_{i, \rho_i}$ on $M_i$ 
with $\varphi_{i,\rho_i}|_{M_i\setminus V_{i,\rho_i}}\equiv 0$ as in (\ref{eq:cut-off fct}). 
For $\rho=(\rho_1, \rho_2)\in{\rm Irr}(K)$ define a function $\varphi_\rho:M\to[0,1]$ by 
$\varphi_\rho:=\varphi_{0,\rho_0} \varphi_{1,\rho_1}$, which gives a cut-off function with 
$\varphi_{\rho}|_{\wtilde M\setminus \wtilde V_{\rho}}\equiv 0$. 
Then we have a Fredholm operator on $L^2(W)^{(\rho)}$ as the deformation 
\[
\hat D_\rho= D+t\varphi_\rho^4 D_K \quad (t\gg 0). 
\]
In particular we have the index 
\[
{\rm index}(\hat D_\rho)=[M](\rho)\in \bbZ. 
\]
On the other hand we have the sum of the deformations 
\[
\hat D_\rho'=(D_0+t\varphi_{0,\rho_0}^4D_{K_0})+\varepsilon_{W_0}(D_1+t\varphi_{1,\rho_1}^4D_{K_1})=
D+t(\varphi_{0,\rho_0}^4D_{K_0}+\varepsilon_{W_0}\varphi_{1,\rho_1}^4D_{K_1}), 
\]which is also Fredholm on $L^2(W)^{(\rho)}$. 
In fact by using the similar estimate in the proof of Proposition~\ref{prop:coercive} 
one can see that $\hat D_\rho'$ is {\it coercive} (see \cite{Anghel} or Proposition~\ref{prop:coercive}) on the outside of a compact subset containing 
$\varphi_{0,\rho_0}^{-1}(0)\cup \varphi_{1,\rho_1}^{-1}(0)=\varphi_\rho^{-1}(0)$. 

\begin{lem}\label{lem:sum=product}
${\rm index}(\hat D_{\rho}')={\rm index}(\hat D_\rho)=[M](\rho)$. 
\end{lem}
\begin{proof}
This follows from the fact that the deformation of $D$ by 
\[
\varphi_{0,\rho_0}^4\varphi_{1,\rho_1}^{4\delta}D_{K_0}+\varepsilon_{W_0}\varphi_{0,\rho_0}^{4\delta}\varphi_{1,\rho_1}^4D_{K_1} 
\quad (0\leq\delta\leq 1)
\]gives a family of coercive operators by using the similar argument in the proof of Proposition~\ref{prop:coercive}.   
\end{proof}

Now consider the Fredholm operator $D_1+t\varphi_{1,\rho_1}^4D_{K_1}$ on $L^2(W_1)^{(\rho_1)}$ 
and we put 
\[
E_{\rho_1}:=\ker(D_1+t\varphi_{1,\rho_1}^4D_{K_1})=E^+_{\rho_1}\oplus E^-_{\rho_1}
\]as the $\bbZ/2$-graded finite dimensional vector space.  
Then there is a natural embedding 
\[
L^2(W_0\otimes E_{\rho_1})^{(\rho_0)}\to L^2(W)^{(\rho)} 
\]whose image is preserved by $(D_0+t\varphi_{\rho_0}^4D_{K_0})\otimes{\rm id}$. 
Let $D_{\rho_0,E_{\rho_1}}$ be the restriction of $(D_0+t\varphi_{\rho_0}^4D_{K_0})\otimes{\rm id}$ on this image, 
which gives a Fredholm operator on $L^2(W_0\otimes E_{\rho_1})^{(\rho_0)}$. 
\begin{prop}\label{prop:productformula1}
We have 
\[
[M](\rho)={\rm index}(D_{\rho_0,E_{\rho1}}). 
\]If we write ${\rm index}(D_0+t\varphi_{\rho_0}^4D_{K_0})=E_{\rho_0}^+-E_{\rho_0}^-$ as an element in 
the K-group K(pt)$\cong\bbZ$, then we have  
\[
[M](\rho)=(E_{\rho_0}^+-E_{\rho_0}^-)\otimes (E_{\rho_1}^+-E_{\rho_1}^-). 
\]
\end{prop}
\begin{proof}
This follows from Lemma~\ref{lem:sum=product} and the fact that 
the above construction satisfies \cite[Assumption~3.14]{Fujita-Furuta-Yoshida2}.  
\end{proof}

Hereafter we exhibit examples and useful formulas. 
These examples give local models in the computation in Section~\ref{sec:Riemann-Roch character of proper Hamiltonian torus space}. 

\begin{example}[Cylinder]\label{ex:cylinder}
Let $M_1$ be the cotangent bundle of the circle $T^*S^1\cong \bbR\times S^1$ equipped with 
the standard symplectic structure, almost complex structure and the natural $S^1$-action on the $S^1$-factor. 
Let $(r,\theta)$ be the coordinate on $M_1$.  
Fix $\rho\in{\rm Irr}(S^1)\cong \bbZ$ and put 
\[
L_\rho:=M_1\times \bbC_{(\rho)}, 
\]where $\bbC_{(\rho)}$ is the one dimensional Hermitian vector space with $S^1$-action of weight $\rho$.  
We take a connection $\nabla$ on $L_\rho$ defined by 
\[
\nabla=d-2\pi\sqrt{-1}\mu(r)dr, 
\] where $\mu:\R\to \R$ is a smooth non-decreasing $S^1$-invariant function such that 
\[
\mu(r)=
\begin{cases} 
r+\rho \quad \left(|r|<\frac{1}{4}\right) \\
\frac{1}{2}+\rho \quad \left(|r|>\frac{3}{4}\right). 
\end{cases}
\]
We take a Clifford module bundle $W_{1,\rho}$ as 
\[
W_{1,\rho}=\wedge^{\bullet}T_{\bbC}M_1\otimes L_\rho=(\bbC\oplus\bbC)\otimes L_\rho, 
\]with the Clifford action $c:T^*M_1\to {\rm End}(W_{1,\rho})$ given by 
\[
c(dr)=
\begin{pmatrix}
0 & -\sqrt{-1} \\ 
-\sqrt{-1} & 0 
\end{pmatrix}, \quad 
c(d\theta)=
\begin{pmatrix}
0 & {-1} \\ 
1 & 0 
\end{pmatrix}. 
\]
These structures give rise a Dolbeault-Dirac operator $D$ and an 
$S^1$-acyclic orbital Dirac-type operator $(D_{1,\rho}, \{V_{1,\rho, \tau}\}_{\tau})$ 
with 
\[
V_{1,\rho, \tau}=
\begin{cases}
M_1\setminus (\{0\}\times S^1) \quad (\tau=\rho) \\ 
M_1 \quad (\tau\neq\rho)
\end{cases}
\] 
and all the data satisfy the condition in Theorem~\ref{prop:sufficientcondD_K}. 
In particular we have the resulting index as an element in $R^{-\infty}(S^1)$. 
We denote it by $[M_{1,\rho}]$. 
By the direct computation one has the following. 
\begin{prop}\label{prop:cylinder}
$[M_{1,\rho}]$ is the delta function supported at $\rho\in{\rm Irr}(S^1)$. Namely we have 
\[
[M_{1,\rho}] : R(S^1)\to \bbZ, \quad \tau\mapsto \delta_{\rho\tau}. 
\]
\end{prop}

\end{example}

\begin{example}[Vector space]\label{ex:disc}
Consider $M_2=\bbC$ with the standard $S^1$-action. 
Let $B_{\delta}(0)$ be the open disc centered at the origin with radius $\delta>0$.  
Here we take an  $S^1$-invariant metric on $M_2$ so that it is standard 
on $B_{\frac{1}{4}}(0)$ and isometric on the outside of $B_{\frac{3}{4}}(0)$ to that on 
the subset $\left\{r\geq\frac{3}{4}\right\}\times S^1$ of  $M_1$. 
Put 
\[
L_\rho:=M_2\times \bbC_{(\rho)}. 
\]
We take a connection $\nabla$ on $L_\rho$ and a Clifford module bundle $W_{2,\rho}$ so that they are 
standard on  $B_{\frac{1}{4}}(0)$ and isomorphic to those on $\left\{r>\frac{3}{4}\right\}\times S^1\subset M_1$ 
in Example~\ref{ex:cylinder} under the identification between $M_2\setminus B_{\frac{3}{4}}(0)$. 
These structures give rise a Dirac operator $D$ and an $S^1$-acyclic orbital Dirac-type operator $(D_{2,\rho}, \{V_{2,\rho,\tau}\}_{\tau})$ 
with 
\[
V_{2, \rho,\tau}=\bbC\setminus\{0\}
\] 
and all the data satisfy the condition in Theorem~\ref{prop:sufficientcondD_K}. 
We denote the resulting index by $[M_{2,\rho}]$. 
By the direct computation one has the following. 
\begin{prop}\label{prop:cylinder2}
$[M_{2,\rho}]$ is the delta function supported at $\rho\in{\rm Irr}(S^1)$. Namely we have 
\[
[M_{2,\rho}] : R(S^1)\to \bbZ, \quad \tau\mapsto \delta_{\rho\tau}. 
\]
\end{prop}

\end{example}

 \begin{example}[Product of cylinders and discs]
 Let $l,m$ be non-negative integers and $M$ the product 
 of $l$ copies of the cylinder $M_1$ and $m$ copies of the disc $M_2$ in the previous examples; 
 \[
 M:=M_1\times\cdots M_1\times M_2\times\cdots \times M_2=(M_1)^l\times (M_2)^m. 
 \]There is the natural induced action of $K:=(S^1)^{l+m}$ on $M$. 
  We use the natural identifications 
 \[
 {\rm Irr}(K)=\left({\rm Irr}(S^1)\right)^{l+m}, 
 \]and 
 \[
 R(K)=R(S^1)^{\otimes (l+m)}. 
 \]
 Take $\rho=(\rho_1, \ldots, \rho_{l},\rho_1' ,\ldots, \rho_{k}')\in{\rm Irr}(K)$ and consider the corresponding structures 
 $(M_1, W_{1,\rho_i}, D_{1,\rho_i}, \{V_{1,\rho_i,\tau}\}_{\tau\in{\rm Irr}(S^1)})$  and 
 $(M_2, W_{2,\rho_j'}, D_{2,\rho_j'}, \{V_{2,\rho_j', \tau}\}_{\tau\in{\rm Irr}(S^1)})$. 
 Using the outer tensor product we can define the product of the Clifford module bundle 
 \[
 W_\rho:=W_{1,\rho_1}\boxtimes \cdots\boxtimes W_{1, \rho_l} \boxtimes  W_{2, \rho_1'} 
 \boxtimes \cdots \boxtimes W_{2, \rho_m'}
 \]which is a Clifford module bundle over $M$.  
 The products 
 \[
 D_K:=D_{1,\rho_1}\boxtimes \cdots\boxtimes D_{1, \rho_l} \boxtimes  D_{2, \rho_1'} 
 \boxtimes \cdots \boxtimes D_{2, \rho_m'}
 \]and 
 \[
 V_\tau:=V_{1,\rho_1, \tau_1}\times \cdots \times V_{1,\rho_l, \tau_l}\times V_{2,\rho_1',\tau_1'}\times 
 \cdots \times V_{2,\rho_m',\tau_m'}  \quad (\tau=(\tau_1,\ldots, \tau_l,\tau'_1,\ldots, \tau_m')\in {\rm Irr}(K))
 \] 
 induce a $K$-acyclic orbital Dirac-type operator on $(M,W)$,  
 where for operators $A:\calH_0\to \calH_0$ and $B:\calH_1\to \calH_1$ on $\bbZ/2$-graded Hilbert spaces their product 
 $A\boxtimes B :\calH_0\otimes\calH_1\to \calH_0\otimes\calH_1$ is defined by 
 \[
 A\boxtimes B:=A\otimes{\rm id}+ \varepsilon_0 \otimes B  
 \]with  the grading operator $\varepsilon_0$ of $\calH_0$.  
In fact the data $(D_K, \{V_\tau\}_\tau)$ satisfy the conditions in Theorem~\ref{prop:sufficientcondD_K}, 
in particular we have the resulting index $[M_\rho]\in R^{-\infty}(K)$. 
The product formula (Proposition~\ref{prop:productformula1}) implies the following equality. 
\begin{prop}\label{prop:product}
We have 
\[
[M_\rho]=[M_{1,\rho_1}]\otimes \cdots \otimes [M_{1,\rho_l}]\otimes [M_{2,\rho_1'}]\otimes \cdots \otimes[M_{2,\rho_m'}].  
\]Namely $[M_{\rho}]$ is the delta function supported at $\rho\in{\rm Irr}(K)$.  
\end{prop}
This structure serves as a local model of a neighborhood of the fiber of the moment map of 
symplectic toric manifold in Section~\ref{sec:A Danilov type formula for proper toric manifold}. 

\end{example}

\subsection{Fiber bundle over a closed manifold}
\label{subsec:Fiber bundle over a closed manifold}
Let $X$ be a closed Riemannian manifold, $E\to X$ a $\bbZ/2$-graded Clifford module bundle over $X$ 
and $P\to X$ a principal $G$-bundle for a compact Lie group $G$.  
Consider a $K$-acyclic orbital Dirac-type operator $(D_K, \{V_{\rho}\}_{\rho\in{\rm Irr}(K)})$ on $(M,W)$ as in Theorem~\ref{prop:sufficientcondD_K}. 
Suppose that $G\times K$ acts on $W\to M$ in an isometric way and 
$(D_K, \{V_{\rho}\}_{\rho\in{\rm Irr}(K)})$ is $G$-invariant.  
Consider the diagonal action of $G$ on $P\times M$ and the quotient manifold 
\[
\wtilde M:=(P\times M)/G, 
\]which has a structure of $M$-bundle $\pi:\wtilde M\to X$. 
Let $\wtilde W\to \wtilde M$ be the vector bundle defined by 
\[
\wtilde W:=\pi^*E\otimes\left((P\times W)/G\right), 
\] which has a structure of a Clifford module bundle over $\wtilde M$ by using an appropriate connection of $P$. 
One can define operators $\wtilde D_W$ and $\wtilde D_E$ on $\wtilde W$ as lifts 
(by using a trivialization of $P$ and a partition of unity if necessary) of Dirac operators $D_W$ on $W$ and $D_E$ on $E$. 
Then 
\[
\wtilde D:=\wtilde D_E+\wtilde D_W
\]is a Dirac operator on $\wtilde W$. 

For $\rho\in{\rm Irr}(K)$ let $\wtilde V_{\rho}$ be the open subset defined by 
\[
\wtilde V_{\rho}:=(P\times V_{\rho})/G 
\]whose complement in $\wtilde M$ is compact.  
$D_K$ induces an operator $\wtilde D_K$ on $\wtilde W$. 
One can see that $(\wtilde D_K, \{\wtilde V_{\rho}\}_{\rho\in {\rm Irr}(K)})$ is a $K$-acyclic orbital Dirac-type operator on 
$(\wtilde M, \wtilde W)$.  
In particular we have a Fredholm operator 
\[
\wtilde D_{\rho}=\wtilde D+t{\wtilde\varphi_\rho}^4 \wtilde D_K
\]on $L^2(\wtilde W)^{(\rho)}$, where $\wtilde\varphi_\rho:\wtilde M\to [0,1]$ is the 
cut-off function induced from the cut off function $\varphi_\rho$ on $M$ as in (\ref{eq:cut-off fct}). 
In this way we have an element $[\wtilde M]\in R^{-\infty}(K)$ defined by 
\[
[\wtilde M](\rho):={\rm index}(\wtilde D_\rho). 
\]

Now consider the Fredholm operator $D_W+t\varphi_{\rho}^4D_{K}$ on $L^2(W)^{(\rho)}$ 
and we put 
\[
E_{\rho}:=\ker(D_W+t\varphi_{\rho}^4D_{K})=E^+_{\rho}\oplus E^-_{\rho}
\]as the $\bbZ/2$-graded finite dimensional vector space.  
Then there is a natural embedding 
\[
L^2(E\otimes E_{\rho})\to L^2(\wtilde W)^{(\rho)} 
\]whose image is preserved by $\wtilde D_E$. 
Let $D_{E, {\rho}}$ be the restriction of $\wtilde D_E$ on this image, 
which gives a Fredholm operator on $L^2(E\otimes E_{\rho})$ 
because the symbol of $D_{E, {\rho}}$ is equal to the tensor product of ${\rm id}_{E_{\rho}}$ and the symbol of $D_E$, 
in particular it is an elliptic operator on the closed manifold $X$. 

\begin{prop}\label{prop:productformula2}
For each $\rho\in{\rm Irr}(K)$ we have 
\[
[\wtilde M](\rho)={\rm index}(D_{E, \rho}). 
\]If we write ${\rm index}(D_E)=E_{0}^+-E_{0}^-$ as an element in 
the K-group K(pt)$\cong\bbZ$, then we have  
\[
[\wtilde M](\rho)=(E_{0}^+-E_{0}^-)\otimes (E_{\rho}^+-E_{\rho}^-). 
\]
\end{prop}
\begin{proof}
This follows from the fact that the above construction satisfies \cite[Assumption~3.14]{Fujita-Furuta-Yoshida2}.  
\end{proof}

\begin{example}\label{ex:localmodel2}
Let $K$ be a torus. 
Consider $M=T^*K$ with the $K$-acyclic orbital Dirac-type operator $(D_K, \{V_{\rho}\}_{\rho\in{\rm Irr}(K)})$ 
defined as the product of Example~\ref{ex:cylinder}. 
Suppose that we take a Clifford module bundle by using $\bbC_{(\rho)}$ for a fixed $\rho\in{\rm Irr}(K)$. 
Then we have  
\[
[M] : R(K)\to \bbZ, \quad \rho'\mapsto \delta_{\rho\rho'}. 
\]
Let $X$ be a closed Riemannian manifold, $E\to X$ a Clifford module bundle and $P\to X$ a principal $K$-bundle. 
Let $\wtilde M$ be the $M$-bundle over $X$ defined by 
\[
\wtilde M=(P\times M)/K. 
\]
Proposition~\ref{prop:productformula2} ensures us that 
\[
[\wtilde M] : R(K)\to \bbZ, \quad \rho'\mapsto {\rm index}(E)\delta_{\rho\rho'}, 
\]where ${\rm index}(E)$ is the index of a Dirac operator on $E$. 
This example serves as a local model of a neighborhood of the 
inverse image of the moment map of Hamiltonian torus action in Section~\ref{subsec:[Q,R]=0 for proper Hamiltonian torus space}. 
\end{example}

\section{Vanishing theorem for fixed points}
\label{sec:A vanishing theorem}
In this section we show the following vanishing theorem for our index, which is a modification of 
\cite[Theorem~6.1]{Fujita-Furuta-Yoshida3} and plays an important role in the subsequent section. 
Though we only use the circle action case in this paper,  we give a slight general version below.

For a torus $K$ we consider a $K$-acyclic orbital Dirac-type operator 
on a Hermitian manifold $M$ with a $K$-equivariant line bundle $L\to M$ as in the end of Section~\ref{subsec:Construction of D_K}. 
We fix and use the Clifford module bundle $W_{\rho}=\wedge^{\bullet} T_{\bbC}M\otimes L\otimes\bbC_{(\rho)}$, 
where $\bbC_{(\rho)}$ is the 1-dimensional irreducible representation of $K$ with weight $\rho$. 
We put the following assumptions. 
\begin{assump}\label{assump:vanish}
Together with the conditions in Theorem~\ref{prop:sufficientcondD_K} we assume the followings.  
\begin{itemize}
\item A compact Lie group $H$ acts on $M$, which commutes with $K$-action and 
all the additional data are $H\times K$-equivariant. 
\item $Z_\rho$ is equal to the fixed point set $M^K$, and it is a closed connected submanifold of $M$. 
\item The fixed point set $L^K$ is equal to the image of $M^K$ in $L|_{M^K}$  by the zero section.  
\end{itemize}
\end{assump}

\begin{thm}\label{thm:vanish}
Under Assumption~\ref{assump:vanish} we have 
\[
[Z_{\rho}]={\rm index}_H(\hat D_{\rho})=0\in R(H). 
\]
\end{thm}

To show it we show a reducing rank lemma. 
Suppose that there exists a subtorus $K'$ of $K$ and $\rho'\in {\rm Irr}(K')$ such that the following conditions are satisfied. 
\begin{itemize}
\item  The restriction of $\rho$ to $K'$-action is $\rho'$, i.e., $\iota_{K'}^*(\rho)=\rho'$. 
\item $Z_{\rho'}={\rm Zero}(\underline{\rho'}-\underline{\mu'})$ is compact for $\mu':=\iota_{K'}^*\circ\mu$.
\item The differential operator 
\[
D_{K'}=\sum_{i=1}^{{\dim}K'}c(\underline{\xi_i^M})(\calL_{\xi_i}-\sqrt{-1}\mu_i)
\]and an open subset $V_{\rho'}:=M\setminus Z_{\rho'}$ give a $\rho'$-acyclic orbital Dirac-type operator on $(M,W_\rho)$.  
\end{itemize}
The deformation $\hat D_{\rho'}=D+t\varphi_{\rho'}^4D_{K'}$ gives a Fredholm operator on 
the isotypic component $L^2(W_\rho)^{(\rho')}$ for $t\gg 0$, 
where $\varphi_{\rho'}$ is a cut-off function for  $V_{\rho'}$ as in (\ref{eq:cut-off fct}). 
On the other hand the condition $\iota_{K'}^*(\rho)=\rho'$ implies that 
$L^2(W_\rho)^{(\rho)}$ is a subspace of $L^2(W\rho)^{(\rho')}$ and $(\hat D_{\rho'})^{(\rho')}$ preserves it. 
We define ${\rm index}(\hat D_{\rho',\rho})$ as its Fredholm index; 
\[
{\rm index}(\hat D_{\rho',\rho}):=
{\rm index}((\hat D_{\rho'})^{(\rho')}:L^2(W_\rho)^{(\rho)}\to L^2(W_\rho)^{(\rho)}). 
\] 
We can incorporate $H$-action and regard them as $H$-equivariant indices ${\rm index}_{H}(\cdot)$. 
\begin{lem}\label{lem:rankreduce}
$[Z_{\rho}]={\rm index}_H(\hat D_{\rho})={\rm index}_H(\hat D_{\rho',\rho})\in R(H)$. 
\end{lem}  
\begin{proof}
By taking a basis of $\frakk$ which is an extension of a basis of $\frakk'$ we may assume that 
\[
D_K=\sum_{i=1}^{{\dim}K}c(\underline{\xi_i^M})(\calL_{\xi_i}-\sqrt{-1}\mu_i)
\]
and 
\[
 D_{K'}=\sum_{i=1}^{{\dim}K'}c(\underline{\xi_i^M})(\calL_{\xi_i}-\sqrt{-1}\mu_i).  
 \]We also define $D_{K,K'}$ by 
\[
D_{K,K'}:=D_K-D_{K'}. 
\]
Take and fix cut-off functions $\varphi_{\rho}$ for $V_\rho$ and $\varphi_{\rho'}$ for $V_{\rho'}$ as in (\ref{eq:cut-off fct}). 
We put  $\varphi_{\rho,\rho'}:=\varphi_\rho\varphi_{\rho'}$. 
There exists $t>0$ such that the deformation 
\begin{equation}\label{eq:deformt}
D+t\varphi_{\rho,\rho'}^4D_K
\end{equation}
gives a Fredholm operator on the isotypic component $L^2(W_\rho)^{(\rho)}$. 
The almost same argument in the proof of Theorem~\ref{thm:B=F} implies that for any $t'\geq t$ 
the deformation 
\[
D+\varphi_{\rho,\rho'}^4(t'D_{K'}+tD_{K,K'})
\]is Fredholm on  $L^2(W_\rho)^{(\rho)}$ and its Fredholm index is same as that of (\ref{eq:deformt}). 
On the other hand for fixed such $t$ the family 
\[
D+\varphi_{\rho,\rho'}^4(t'D_{K'}+\epsilon tD_{K,K'}) \quad (\epsilon\in [0,1])
\]satisfies the coercivity on the interior of $\varphi_{\rho,\rho'}^{-1}(1)$ for $t'\geq t$ large enough. 
It implies 
\begin{eqnarray*}
{\rm index}_H(D+t'\varphi_{\rho,\rho'}^4D_{K'})  
&=&  {\rm index}_H(D+\varphi_{\rho,\rho'}^4(t'D_{K'}+ tD_{K,K'})) \\
&=&  {\rm index}_H(D+t\varphi_{\rho,\rho'}^4D_{K}). 
\end{eqnarray*}
The excision property implies 
\[
[Z_\rho]={\rm index}_H(D+t\varphi_{\rho}^4D_{K})  ={\rm index}_H(D+t\varphi_{\rho,\rho'}^4D_{K}) 
\]and 
\[
{\rm index}_H(\hat D_{\rho',\rho})=
{\rm index}_H(D+t'\varphi_{\rho'}^4D_{K'})  
={\rm index}_H(D+t'\varphi_{\rho,\rho'}^4D_{K'}),   
\]which complete the proof. 
 \end{proof}

\begin{Rem}
To show Lemma~\ref{lem:rankreduce} we do not use the assumption $Z_{\rho}=M^K$.  
\end{Rem}

\begin{prop}\label{prop:vanishvect}
Theorem~\ref{thm:vanish} is true when $M$ is a small open disc around the 
origin of a Hermitian vector space on which the $K$-action is linear and $M^K$ consists of the origin. 
\end{prop}
\begin{proof}
By considering the tensor product it suffices to prove in the case that $\rho$ is the trivial representation ${\bf 0}$.   
We can choose an appropriate generic circle subgroup $K_1$ of $K$ so that $K_1$ acts on $M$ with $M^{K_1}=\{0\}$ and 
the $K_1$-action on $L|_0$ is nontrivial. 
In fact let $\rho_1,\ldots, \rho_{\dim M}\in {\rm Irr}(K)$ be the weights appeared in the linear action on $M$, 
all of which are non-zero by the assumption $M^{K}=\{0\}$, then we can take a splitting of the differential 
of the representation $K\to U(1)$ on $L|_0$ such that the image of the splitting in ${\frakk}$ is rational and 
is not perpendicular to any $\rho_i$. The subgroup of the image gives the desired circle subgroup. 
By Lemma~\ref{lem:rankreduce} we have 
\[
{\rm index}(\hat D_{\bf 0})={\rm index}(D+t\varphi_{\bf 0}^4 D_K)={\rm index}(D+t\varphi_{\bf 0}^4 D_{K_1})\in\bbZ. 
\]On the other hand \cite[Proposition~6.8]{Fujita-Furuta-Yoshida3} and Theorem~\ref{thm:FFY=F} imply 
 \[
 {\rm index}(D+t\varphi_{\bf 0}^4 D_{K_1})=0, 
 \]and we complete the proof. 
\end{proof}

\begin{proof}[Proof of Theorem~\ref{thm:vanish}] 
The claim follows from Proposition~\ref{prop:vanishvect} and the product formula (Proposition~\ref{prop:productformula2})
with the same argument in \cite[Section~6.4]{Fujita-Furuta-Yoshida3}. 
\end{proof}
  
  \section{Quantization of non-compact Hamiltonian torus manifolds} 
  \label{sec:Riemann-Roch character  of proper Hamiltonian torus space}
  In this section by using the ingredients established in the previous sections 
  we define quantization of non-compact symplectic manifolds 
 equipped with Hamiltonian group action and show [Q,R]=0 for circle action case and a Danilov-type fomula for toric action case. 
  
  \subsection{Definition : general case}
  Let $K$ be a compact torus and $M$ a symplectic manifold equipped with Hamiltonian $K$-action. 
  Suppose that there exists a $K$-equivariant prequantizing line bundle $(L,\nabla)$ 
  and let $\mu:M\to \frakk^*$ be the associated moment map.  
  We use the Clifford module bundle $W=\wedge^\bullet T_\bbC M\otimes L$ for a $K$-invariant compatible almost complex structure. 
  We assume the following for the moment : 
  \begin{assump}\label{assump:Zcompact}
  For each $\rho\in {\rm Irr}(K)$ the zero set   $Z_{\rho}={\rm Zero}(\underline\rho-\underline\mu)$ is  compact.   
  \end{assump}
  
  \begin{dfn}\label{dfn:equivindHam}
  We define its quantization $\calQ_K(M)\in R^{-\infty}(K)$ by 
  \begin{equation}
  \calQ_K(M)(\rho):=[\wtilde X_{\rho}](\rho) \in \bbZ \quad (\rho\in {\rm Irr}(K)), 
  \end{equation}
  where $\wtilde X_{\rho}$ is a complete manifold containing $Z_{\rho}$ as its neighborhood 
  on which the Dirac operator along orbits defined as in Definition~\ref{dfn:exD_K}   
  gives a $\rho$-acyclic orbital Dirac-type operator for $\wtilde X_\rho\setminus Z_{\rho}$. 
  \end{dfn}
 
  The excision property guarantees that the number $\calQ_K(M)(\rho)$ is independent from the choice of such $\wtilde X_{\rho}$. 
  Theorem~\ref{thm:localizationformula} enable us to describe $\calQ_K(M)(\rho)$ into the sum of local contributions 
  \[
  \calQ_K(M)(\rho)=[\mu^{-1}(\rho)]+\sum_{\alpha\in{\rm Irr}(K)\setminus\{\rho\}}[Z_{\rho, \alpha}]. 
  \]It would be natural to expect the vanishing of $[Z_{\rho, \alpha}]$. 
  One possible way to show this vanishing is using a combination of the coincidence of $[Z_{\rho,\alpha}]$ with the transverse index and 
  vanishing results for it, e.g., by Paradan \cite{ParadanlocRR}. 
  In the subsequent subsections, instead of using them, we have the vanishing of $[Z_{\rho, \alpha}]$ 
  for the circle action case and toric case
  based on Theorem~\ref{thm:vanish}, and we define the quantization $\calQ_K(M)$ under a weaker assumption than 
  Assumption~\ref{assump:Zcompact}.

 The quantization $\calQ_K(M)$ is a generalization of $K$-equivariant spin$^c$ quantization using the index of 
 Dolbeault-Dirac operator 
  in the compact case, which is often denoted by $RR_K(M)$ and  
  called the {\it equivariant Riemann-Roch number} or {\it Riemann-Roch character}.

  \subsection{[Q,R]=0 for non-compact Hamiltonian torus manifolds}
  \label{subsec:[Q,R]=0 for proper Hamiltonian torus space}
  In this subsection we consider the case $K=S^1$. 
  Since in this case one has 
  \[
  Z_{\rho}=\mu^{-1}(\rho)\cup M^{K}, 
  \]for each $\rho \in {\rm Irr}(K)=\Lambda^*$ and $\mu(M^{K})\subset \Lambda^*$
   the quantization $\calQ_K(M)$ has a localization property to $\Lambda^*$. 
   Moreover one has a decomposition 
   \[
   M^K=\bigcup_{\alpha\in\Lambda^*} M^K\cap \mu^{-1}(\alpha)
   \]which gives us a decomposition of the index 
   \[
   [Z_{\rho}]=[\mu^{-1}(\rho)]+\sum_{\alpha\in\Lambda^*\setminus\{\rho\}}[M^{K}\cap\mu^{-1}(\alpha)]\in\bbZ, 
   \] where we use the notation as in Section~\ref{subsec:Non-complete case and localization formula}.  
   Proposition~\ref{prop:BS->Z} and  Theorem~\ref{thm:vanish} implies that we have 
   \[
  [M^{K}\cap\mu^{-1}(\alpha)]=0 \quad (\alpha\in\Lambda^*\setminus\{\rho\}). 
  \] This observation enable us to define $\calQ_K(M)$ by 
  \[
  \calQ_K(M)(\rho):=[\mu^{-1}(\rho)] 
  \]without Assumption~\ref{assump:Zcompact}.  
  We only need the following assumption : 
  \begin{assump}\label{assump:properness}
  The  preimage  of  each lattice   point  in $\Lambda^*$  is compact. 
  \end{assump}
   This definition leads us to a proof of [Q,R]=0, the principal of \lq\lq quantization commutes with reduction\rq\rq, 
   as in \cite{Fujita-Furuta-Yoshida3} in the non-compact case. 
 
 For a regular value $\xi\in{\frakk}^*$ of $\mu:M\to{\frakk}^*$ let $M_\xi$ be the symplectic quotient at $\xi$: 
 \[
 M_\xi:=\mu^{-1}(\xi)/K, 
 \] which is a closed symplectic manifold (orbifold) under Assumption~\ref{assump:properness}. 
 Moreover if a regular value $\rho$ is an element of ${\rm Irr}(K)$, then there exists a natural prequantizing line bundle 
 over $M_\rho$, and hence, one can define the Riemann-Roch number $RR(M_\rho)$ as the 
 index of the Dolbeault-Dirac operator associated with a $K$-invariant compatible almost complex structure.

  \begin{thm}\label{noncpt[Q,R]=0}
  Suppose that $\rho\in{\rm Irr}(K)$ is a regular value of the moment map $\mu:M\to {\frakk}^*$. 
  Then we have 
  \[
\calQ_K  (M)(\rho)=RR(M_\rho). 
  \]
  \end{thm}
  \begin{proof}
  A neighborhood of $\mu^{-1}(\rho)$ in $M$ can be identified with the product 
  \[
  (T^*K\times \mu^{-1}(\rho))/K
  \]by the Darboux-type theorem (see \cite[Lemma~7.1]{Fujita-Furuta-Yoshida3} for example), 
  which has a structure of $T^*K$-bundle over $M_\rho$. 
  By applying the product formula in Example~\ref{ex:localmodel2} we have 
  \[
  [\mu^{-1}(\rho)]=RR(M_\rho). 
  \]
  \end{proof}

  \begin{Rem}
  \begin{enumerate}
  \item Even for a higher rank torus case, by choosing a circle subgroup generic enough one can give a proof of Theorem~\ref{noncpt[Q,R]=0} 
  by induction. 
  \item Due to Corollary~\ref{cor:F=Transverse} the quantization $\calQ_K(M)$ can be identified with 
  Atiyah's transverse index. 
  Theorem~\ref{noncpt[Q,R]=0} gives an alternative proof of Vergne's conjecture for torus case to Ma-Zhang's proof  
  in \cite{Ma-Zhangarxiv} which uses Braverman's deformation. 
  \item The above construction and a proof of Theorem~\ref{noncpt[Q,R]=0} is essentially same as those in \cite{FujitaS1}. 
  \end{enumerate}
  \end{Rem}

  \subsection{A Danilov-type formula for non-compact toric manifolds}
  \label{sec:A Danilov type formula for proper toric manifold}
   Now we focus on the {\it symplectic toric} case. Namely we assume that $K$ is a torus with $2\dim(K)=\dim(M)$. 
   In this case Assumption~\ref{assump:properness} is automatically satisfied because the preimage of each point is a single orbit. 
 We can define the quantization $\calQ_K(M)$ as it is noted in the previous section.  
 In fact for each $\rho\in{\rm Irr}(K)$ the image  $\mu(Z_{\rho,\alpha})$ of the component of $Z_{\rho}$  is 
 contained in the boundary of the momentum polytope $\mu(M)$, and one can see $[\mu^{-1}(\rho)]=1$  and $[Z_{\rho,\alpha}]=0$ 
 by the same argument in \cite[Section~6.1]{Fujitaorigami} 
 together with Proposition~\ref{prop:BS->Z} and  Theorem~\ref{thm:vanish}.  
 These observations enable us to define $\calQ_K(M)\in R^{-\infty}(K)$ and give the following description, 
 which is a non-compact generalization of Danilov's formula.

  \begin{thm}\label{thm:noncptDanilov}
  Under the above set-up we have 
  \[
  \calQ_K(M)=\sum_{\rho\in \mu(M)\cap \Lambda^*}\bbC_{(\rho)},  
  \]where the right hand side is an element in $R^{-\infty}(K)$ which is characterized by  
  \[
  {\rm Irr}(K)\ni \rho' \mapsto 
  \begin{cases}
  1 \quad (\rho'\in \mu(M)\cap \Lambda^*) \\
  0 \quad (\rho'\notin\mu(M)\cap \Lambda^*). 
  \end{cases}
  \] 
  \end{thm}
  
  \begin{Rem}
  In a general framework of geometric quantization one uses an additional structure called a {\it polarization}, 
  which is an integrable Lagrangian distribution of the complexification of the tangent bundle. 
  One typical example is a {\it K\"ahler polarization} which is defined as a compatible complex structure. 
  Our quantization is the spin$^c$ quantization, which is a quantization based on a polarization relaxed the integrality condition 
  in the K\"ahler polarization.   
  The quantization is given by the Fredholm index of the Dolbeault-Dirac operator. 
  The other example is a {\it real polarization}, which is defined by the tangent bundle along fibers of the Lagrangian fibration. 
  In the real polarization case it is known that the quantization can be described by {\it Bohr-Sommerfeld fibers}, which are 
  characterized by the existence of non-trivial global parallel sections of the prequantizing line bundle on the orbits.  
  The moment map of toric manifolds can be regarded as a real polarization with singular fibers.  
  In the toric case, the Bohr-Sommerfeld fibers are nothing other than the inverse images of the integral lattice points in the 
  momentum polytope. 
  
  One important topic in geometric quantization is the problem of independence from the polarizations.  
  There are several results supporting the coincidence between the quantizations obtained by the spin$^c$ polarization 
  and the real polarization from the view point of index theory, such as \cite{Andersen, Fujita-Furuta-Yoshida1, Kubota,Yoshida100}.  
  Theorem~\ref{thm:noncptDanilov} can be considered as a non-compact version of the above results. 
  \end{Rem}

  \begin{Rem}
  In \cite{Fujitaorigami} we gave a proof of Danilov's formula for compact symplectic toric manifolds 
  (or more generally for {\it toric origami manifolds}) using 
  a localization formula based on the theory of the {\it 
  acyclic compatible  fibration/system} developed in \cite{Fujita-Furuta-Yoshida2}. 
  Since one can see that the acyclic compatible fibration constructed on a given toric manifold 
  does not have a product structure in general, we cannot apply the product formula and 
  have to compare the resulting index with the index of the product. 
  One remarkable difference in the computation of the local contribution is that 
  our deformation by $D_K$ fits into the local product structure of a neighborhood of $\mu^{-1}(\rho)$. 
  In particular we can apply the product formula directly.   
  \end{Rem}
  
 \section{Comments and further discussions}
 \label{sec:Comments and further discussions}
 \subsection{Application to quantization of Hamiltonian loop group spaces}
 \label{subsec:Application to quantization of Hamiltonian loop group spaces}
 Quantization of Hamiltonian loop group spaces is studied in various directions. 
 In particular Loizides-Song \cite{LoizidesSong} studied it from the view point of index theory and KK-theory.  
 Their construction is based on their previous work \cite{Loizides-Song-Meinrenken} with Meinrenken 
 in which they constructed a spinor bundle over a proper Hamiltonian loop group space and 
 a nice finite dimensional non-compact submanifold in it, which is transverse to the orbits of the loop group action. 
 One key ingredient in \cite{LoizidesSong} is to associate a K-homology cycle to such a non-compact manifold. 
 They established an index theory using the ${\rm C}^*$-algebraic condition which they call the {\it $(\Gamma, K)$-admissibility}, 
 where $K$ is a compact Lie group and $\Gamma$ is a countable discrete group with proper length function. 
 They showed that in the proper Hamiltonian loop group space case the $(\Lambda, T)$-admissibility is 
 satisfied for a maximal torus $T$ of $K$, and 
 the resulting K-homology class has an anti-symmetric property with respect to some Weyl group action of $K$, 
 which gives rise quantization as an element in the fusion ring of $K$.

 In this paper we constructed  a similar K-homology cycle without using $(\Gamma, K)$-admissibility. 
 In the subsequent research we will investigate an approach of quantization of Hamiltonian loop group spaces  
 by incorporating the action of the integral lattice $\Lambda$ in our construction appropriately. 
 In such an approach it would be interesting to understand how the localization phenomenon of our index is 
 reflected in the quantization of loop group spaces. 
 
There is an another related work by Takata. 
In \cite{Takata} an $LS^1$-equivariant index is constructed as an element in the fusion ring 
from the view point of KK-theory and non-commutative geometry. 
He also developed an index theorem in infinite dimensional setting in \cite{Takata3, Takata2}. 
 It would be also interesting to investigate how our construction is positioned in Takata's theory. 
 
 \subsection{Deformation as  KK-products}
 \label{subsec:Deformation as  KK-products}
 Motivated by the pioneering work by Kasparov \cite{Kasparov},  
 Loizides-Rodsphon-Song showed  in \cite{LoizidesRodsphonSong} that the K-homology class obtained by 
 Braverman's deformation factors as a KK-product of the Dirac class and a KK-class arising from the 
 deformation. 
 It is desirable to understand our deformation using the acyclic orbital Dirac-type operator as a KK-product. 

\appendix

\section{Appendix : $K$-acyclic orbital Dirac-type operator}\label{sec:Set-up and main assumptions}
In this appendix we give a general machinery to have an equivariant index and a K-homology cycle 
using a deformation by differential operator along orbits. 
Though the machinery itself works for non-abelian case, we do not know good example except $D_K$ as in 
Definition~\ref{dfn:exD_K} for the moment. 
 
 \subsection{Set-up and definition}
 \label{subsec:Set-up and definition}

 Let $M$ be a complete Riemannian manifold and 
 $W\to M $ a $\bbZ/2$-graded ${\rm Cl}(TM)$-module bundle with the Clifford multiplication $c:TM\cong T^*M\to {\rm End}(W)$.  
 
 Let $K$ be a compact Lie group acting on $M$ in an isometric way. 
 We assume that the $K$-action lifts to a unitary action of $W$.   
 Take and fix a $K$-invariant Dirac-type operator $D:\Gamma_c(W)\to \Gamma_c(W)$. 
 
\begin{dfn}[$\rho$-acyclic and $K$-acyclic orbital Dirac-type operator]\label{dfn:K-acyclic orbital system}
Let $\rho$ be an element of ${\rm Irr}(K)$.  A pair 
 \[
( D_K, V_\rho)
 \]is called a {\it $\rho$-acyclic orbital Dirac-type operator on $(M,W)$} if the following conditions are satisfied. 
\begin{enumerate}
 \item $D_K:\Gamma_c(W)\to \Gamma_c(W)$ is a $K$-invariant first order self-adjoint differential operator such that : 
 \begin{enumerate}
 \item $D_K$ contains only differentials along $K$-orbits, and the restriction to each $K$-orbit is an elliptic operator on the orbit. 
 \item  $D_K$ is finite propagation speed, i.e., the principal symbol $\sigma(D_K):T^*M\to {\rm End}(W)$ satisfies 
 \[
\sup \{\|\sigma(D_K)(v)\| \ | \ v\in T^*M, \ |v|=1 \} < \infty. 
 \]
 \item $D_K$ anti-commutes with the Clifford multiplication of the transverse direction to orbits. 
 Namely for any $K$-invariant function $h$ on $M$ one has 
\[
D_Kc(dh)+c(dh)D_K=0. 
\]
\item\label{cond:D_Kbdd} 
The isotypic component $D_K^{(\rho)}$ gives a bounded operator on $L^2(W)^{(\rho)}$. 
 \end{enumerate}  
 \item 
 $V_{\rho}$ is an open subset of $M$ such that $M\setminus V_\rho$ is compact. 
 \item\label{cond:K-acyclic}  
 We have 
 \[
 \ker(D_{K}|_{V_\rho})^{(\rho)}=0. 
 \]
  \item\label{cond:C_pi} 
 There exists a constant
 \footnote{The third condition implies that $(D_K^2)^{(\rho)}$ is a strictly positive operator on each $K$-orbit. 
 On the other hand as in Lemma~\ref{lem : lemonD_K} the condition (c) implies that $DD_K+D_KD$ is also a differential operator along the orbits, 
 we can take such a constant $C_\rho$ for each orbit. This condition means that 
 we can take such constants uniformly on $V_\rho$. } $C_{\rho}>0$ such that 
  \[
  |( (DD_{K}+D_{K}D)s, s )_W|\leq C_{\rho}( D_{K}^2 s,s)_W
  \]and 
  \[
  |( D_{K}s, s )_W|\leq C_{\rho}( D_{K}^2 s,s)_W
  \]
hold  for any $s\in \Gamma_c(W|_{V_\rho})^{(\rho)}$.

  \item\label{cond:eigenposi} 
  There exists a constant $\kappa_\rho>0$ such that 
  \[
  \kappa_\rho( s, s)_W \leq ( D_{K}^2 s,s)_W
  \]
  holds for any $s\in \Gamma_c(W|_{V_\rho})^{(\rho)}$. 
 \end{enumerate} 
 If a family of open subsets $\{V_{\rho}\}_{\rho\in{\rm Irr}(K)}$ gives a $\rho$-acyclic orbital Dirac-type operator $(D_K,V_\rho)$
 for each $\rho\in {\rm Irr}(K)$, 
 then we call $(D_K, \{V_{\rho}\}_{\rho\in{\rm Irr}(K)})$ the {\it $K$-acyclic orbital Dirac-type operator}. 
 \end{dfn}

 The completeness of $M$ implies that there exists a $K$-invariant smooth proper function $f:M\to[1,\infty)$ such that 
 \[
\|df\|_{\infty}:=\sup_{x\in M}|df_x|<\infty.  
 \]
 We take and fix such $f$. 
 For each $\rho\in{\rm Irr}(K)$ we take and fix a $K$-invariant cut-off function 
 \begin{equation}\label{eq:cut-off fct}
 \varphi_\rho:M\to [0,1]
 \end{equation} such that 
 \[
 \varphi_\rho\equiv 0 \ {\rm on \ a \ sufficiently \ small \  compact \ neighborhood \ of} \ M\setminus V_\rho 
 \]
 and 
 \[
 \varphi_\rho\equiv 1 \ {\rm on \ the \ complement \ of \ a \ relatively \ compact \ neighborhood \ of} \  M\setminus V_\rho. 
\]
 We put $f_\rho:=\varphi_\rho f$. 
We consider the deformation of $D$ defined by 
 \[
\hat D_\rho:=D+f_{\rho}^2D_Kf_{\rho}^2=D+f_\rho ^4 D_{K}.  
 \]
Since $D$ and $D_K$ has finite propagation speed,  $\hat D_\rho$ is an 
essentially self-adjoint operator on $L^2(W)$.  
Moreover one can see that $\hat D_\rho$ is {\it transversally elliptic} in the sense of Atiyah  \cite{Atiyah1}. 
In fact for any $K$-invariant function $h:M\to \R$, since $D_K$ commutes with the multiplication by $h$ one has  
\[
\hat D_\rho h -h \hat D_\rho= Dh - hD =c(dh), 
\]which is invertible unless $dh=0$.

Hereafter we mainly consider the isotypic component $\hat D_{\rho}^{(\rho)}$. 
Even if so we often omit the superscript $(\cdot)^{(\rho)}$ of the isotypic component for simplicity and 
use the notation as $\hat D_{\rho}:L^2(W)^{(\rho)}\to L^2(W)^{(\rho)}$ and so on. 

\begin{Rem}
The Clifford module structure and Dirac-type condition are not so essential. 
In fact we can establish almost all propositions, definitions,  etc.,  below for more general 
vector bundles and elliptic operators with finite propagation speed. However since we do not have 
applications of such generalizations we only handle with Clifford module bundles and 
Dirac-type operators in the present paper. 
\end{Rem}

 \subsection{Compactness and $K$-Fredholm property}
 \label{subsec:Compactness and $K$-Fredholmness}
 Let $(D_K, \{V_{\rho}\}_{\rho\in{\rm Irr}(K)})$ be a $K$-acyclic orbital Dirac-type operator on $(M,W)$. 
 We take and fix a family of functions $\{f,\{\varphi_\rho\}_{\rho\in{\rm Irr}(K)}\}$ as above.  
 \begin{prop}\label{prop:compactresolvent}
For each $\rho\in{\rm Irr}(K)$ 
there exists a smooth $K$-invariant proper function $\Phi_{\rho}:M\to \bbR$ such that 
 $\Phi_{\rho}$ is bounded below and we have 
 \[
 (\hat D^2_\rho)^{(\rho)}+1\geq (D^2)^{(\rho)}+\Phi_{\rho} 
 \]as self-adjoint operators on $L^2(W)^{(\rho)}$. 
 \end{prop}
 
 \begin{proof}

 Since $f_\rho$ is $K$-invariant we have an equality on $\Gamma_c(W)^{(\rho)}$ ; 
 \begin{eqnarray*}
 \hat D^2_\rho&=&D^2+(Df_\rho^4 D_{K}+f_\rho^4 D_{K}D)+f_\rho^8D_{K}^2 \\ 
 &=& D^2+f_\rho^{2}(DD_{K}+D_{K}D)f_{\rho}^{2}+c(df_\rho^{2})D_{K}f_{\rho}^{2}-D_{K}f_{\rho}^{2}c(df_\rho^{2})+f_\rho^8D_{K}^2 \\
 &=& D^2+f_\rho^{2}(DD_{K}+D_{K}D)f_{\rho}^{2}+2c(df_\rho^{2})D_{K}f_{\rho}^{2}+f_\rho^8D_{K}^2.  
 \end{eqnarray*}
 Now for any $s\in \Gamma_c(W)^{(\rho)}$ we have 
 \begin{eqnarray*}
| ( f_\rho^{2}(DD_{K}+D_{K}D)f_{\rho}^{2}s, s)_W |&=&
| ( (DD_{K}+D_{K}D)f_{\rho}^{2}s, f_\rho^{2}s)_W | \\
 &\leq& C_\rho( D_K^2f_{\rho}^{2}s, f_\rho^{2}s)_W  \\
 &=&C_\rho( f_\rho^4 D_K^2 s,s)_W
 \end{eqnarray*}
 and 
\begin{eqnarray*}
 |(c(df_\rho^{2})D_{K}f_{\rho}^{2}s,s)_W|&=&|(c(df_\rho ^2)D_{K}f_\rho s, f_\rho s)_W | \\ 
 &=&|(2f_\rho  c(df_\rho )D_{K}f_\rho s, f_\rho s)_W | \\
 &\leq& 2\|df_\rho \|_{\infty} |( D_{K}(f_\rho )^{3/2}s, (f_\rho )^{3/2}s)_W | \\
 &=& 2C_\rho\|df_\rho \|_{\infty} (f_\rho ^3 D_{K}^2s, s)_W . 
 \end{eqnarray*}
 Summarizing the above inequalities we have 
\begin{eqnarray*}
\hat D^2_\rho &\geq&  D^2 +(-C_\rho f_{\rho}^{4}-4C_\rho\|df_\rho \|_{\infty} f_\rho ^3 + f_\rho^8)D_K^2 \\
&\geq& D^2 +\frac{\kappa_\rho f_\rho^8}{2}+\left(-C_\rho f_{\rho}^{4}-4C_\rho\|df_\rho \|_{\infty} f_\rho ^3 + \frac{f_\rho^8}{2}\right)D_K^2. 
\end{eqnarray*}
 Now put 
 \[
g_{\rho}:=\frac{f_\rho^8}{2} -C_\rho f_{\rho}^{4}-4C_\rho\|df_\rho \|_{\infty} f_\rho ^3 : M\to \bbR. 
 \]
 Since $f_\rho$ is proper and bounded below the function 
 $g_\rho$ is also proper and bounded below. 
 Note that $M_-:=g_\rho^{-1}((-\infty, 0])$ is a compact subset of $M$, and hence, by the boundedness of $D_K$ 
 (1.(d) in Definition~\ref{dfn:K-acyclic orbital system})  there exists a constant $C_{\rho, M_-}>0$ such that we have 
 \[
 \int_{M_-} \langle D_K^2s, s\rangle_W\leq C_{\rho,M_-} \int_{M_-} \langle s, s\rangle_W
 \]
 and 
 \begin{eqnarray*}
 ( g_{\rho} D_{K}^2s,s)_W&=&\left(\int_{M_-}+\int_{M\setminus M_-}\right)\langle g_\rho D_{K}^2s,s\rangle_W  \\
 &\geq& \int_{M_-}\langle g_\rho D_{K}^2s,s\rangle_W  \\ 
 &\geq& \min_{M_-}(g_\rho)C_{\rho,M_-}( s,s)_W. 
 \end{eqnarray*}
As a consequence we have 
\[
\hat D_{\rho}^2 +1\geq D^2 +\Phi_{\rho}
\]
 for 
 \[
 \Phi_\rho:=\frac{\kappa_\rho f_\rho^8}{2}+\min_{M_-}(g_\rho)C_{\rho,M_-}+1
 \]which is $K$-invariant, proper and bounded below. 
  \end{proof}
  
  As a corollary we have the following compactness by \cite[Proposition~B.1]{LoizidesSong}. 
 
 \begin{cor}\label{cor:K-Fredholm}
 For any $\rho\in{\rm Irr}(K)$, a bounded operator $((\hat D_\rho^2)^{(\rho)}+1)^{-1}$ on $L^2(W)^{(\rho)}$ 
  is a compact operator. 
  In particular $(\hat D_\rho)^{(\rho)}$ is a Fredholm operator on $L^2(W)^{(\rho)}$. 
 \end{cor}
 
 \begin{dfn}\label{dfn:[M]}
 Define an element $[\hat D]\in R^{-\infty}(K)$ by 
 \[
 [\hat D](\rho):={\rm index}((\hat D_{\rho})^{(\rho)})\in\bbZ 
 \]for each $\rho \in {\rm Irr}(K)$. 
 We also use the notations  
 \[
 [\hat D]=[M,W, D_K]=[M,W]=[M]. 
 \]
 \end{dfn}
 
 Hereafter we often write $[\hat D](\rho)={\rm index}(\hat D_{\rho})\in\bbZ $
  instead of ${\rm index}((\hat D_{\rho})^{(\rho)})$. 
 
 In general a $K$-equivariant operator $A$ on a $\bbZ/2$-graded Hilbert space ${\calH}$ with isometric $K$-action 
 is called {\it $K$-Fredholm} if each isotypic component $A^{(\rho)}:{\calH}^{(\rho)}\to {\calH}^{(\rho)}$ 
 is Fredholm. 
 Such a $K$-Fredholm operator $A$ defines an element in $R^{-\infty}(K)$ denoted by a formal expression;
 \[
 {\rm index}_K(A)=\sum_{\rho\in{\rm Irr}(K)}{\rm index}(A^{(\rho)})\rho.  
 \]
 Corollary~\ref{cor:K-Fredholm} and Definition~\ref{dfn:[M]} imply that 
 \[
 \bigoplus_{\rho\in{\rm Irr}(K)}\hat D_{\rho} :L^2(W) \to L^2(W)
 \]is a $K$-Fredholm operator and $[\hat D]$ is its index in $R^{-\infty}(K)$. 
 
 \subsection{K-homology cycle representing the class $[\hat D]$}
 \label{subsec:K-homology cycle representing the class [ D]}
 We consider the same set-up as in the previous sections. 
 For each $\rho\in{\rm Irr}(K)$ we put 
 \[
 F_{\rho}:=\frac{\hat D_\rho}{\sqrt{1+(\hat D_{\rho})^2}}
 \]which is a bounded operator acting on $L^2(W)^{(\rho)}$ with $\|F_{\rho}\|=1$. 
 We can see that 
 \begin{equation}\label{eq:cycleF}
 F:=\bigoplus_{\rho\in\text{Irr}(K)}F_{\rho}
 \end{equation}gives a bounded operator on $\displaystyle L^2(W)=\bigoplus_{\rho\in\text{Irr}(K)}L^2(W)^{(\rho)}$. 

It is known that the formal completion $R^{-\infty}(K)$ can be identified with the K-homology group of the 
group C$^*$-algebra 
${\rm K}^0(C^*(K))$, 
which is also identified with the KK-group 
${\K\K} (C^*(K),\bbC)$.  
These groups are generated by triples consisting of 
a Hilbert space, a C$^*$-representation of $C^*(K)$ and a bounded operator on the Hilbert space 
satisfying certain boundedness and compactness.  
See \cite{Blackadar, Higson-Roe, Kasparov}  for basic definitions on K-homology or KK-theory. 
The above Corollary~\ref{cor:K-Fredholm} implies the following.  
\begin{prop}\label{prop:K-homologycycle}
 The bounded operator $F$ as in (\ref{eq:cycleF}) together with the natural representation of $C^*(K)$ on $L^2(W)$ 
  gives a {\K}-homology cycle which represents $[\hat D]$; 
 \[
[ (L^2(W), F)]=[\hat D]\in {\K\K} (C^*(K),\bbC)={\rm K}^0(C^*(K))=R^{-\infty}(K). 
 \]
 \end{prop}

  \subsection{Relation with Fujita-Furuta-Yoshida's deformation}
  \label{subsec:Relation with Fujita-Furuta-Yoshida type deformation}
  In this section we consider an another deformation of the form  
  \[
  D_{\rho, t}:=D+t\varphi_\rho^4 D_K \quad (t\geq 0)
  \]for $\rho\in{\rm Irr}(K)$ using a $K$-acyclic orbital Dirac-type operator $(D_K, \{V_{\rho}\}_{\rho\in{\rm Irr}(K)})$, 
  where $\varphi_{\rho}$ is the cut-off function as in (\ref{eq:cut-off fct}). 
  This type of deformation was studied for an {\it acyclic compatible system} in 
 a series of papers \cite{Fujita-Furuta-Yoshida1, Fujita-Furuta-Yoshida2,  Fujita-Furuta-Yoshida3}. 
 One main difference\footnote{In fact the acyclic compatible system is a family of Dirac-type operators along the fibers 
 which is defined on a family of open subsets. The deformation is given by the sum of them by 
 using a partition of unity.  It is one remarkable feature that the acyclic compatible system do not rely on a group action. 
 Though in this paper we do not investigate any relation between the equivariant acyclic compatible system 
 and the $K$-acyclic orbital Dirac-type operator 
 we believe that they give the same index under a suitable assumptions.} 
 between the above deformation and $\hat D_{\rho}$ is the presence of a proper function $f$. 
 To compare them we introduce a 1-parameter family 
  \[
  \bbD_\epsilon
  =D+(1-\epsilon)f_\rho^4 D_K+\epsilon t \varphi_\rho^4 D_K
  =D+ ((1-\epsilon)f^4 +\epsilon t)\varphi_\rho^4 D_K \quad (\epsilon\in[0,1])
  \]which acts on $L^2(W)$. 
  We show the following. 
  
  \begin{thm}\label{thm:FFY=F}
  For each $\rho\in{\rm Irr}(K)$ there exists $t_{\rho}>0$ 
  such that $\{\bbD_{\epsilon}\}_{\epsilon\in[0,1]}$ gives a family of Fredholm operator on $L^2(W)^{(\rho)}$ 
  for any $t>t_\rho$ and its Fredholm index does not depend on $\epsilon$ and $t$. 
  In particular we have  
  \[
  {\rm index}((D_{\rho, t})^{(\rho)})={\rm index}((\hat D_{\rho})^{(\rho)})\in \bbZ. 
  \]
  \end{thm}
  \begin{cor}\label{cor:FFY=F}
  Define $[D_t]\in R^{-\infty}(K)$ by
  \[
  [D_t](\rho):={\rm index}((D_{\rho, t})^{(\rho)}) \quad (t > t_\rho)
  \]for each $\rho\in {\rm Irr}(K)$.   
  Then we have 
  \[
  [D_t]=[\hat D]\in R^{-\infty}(K).  
  \]
  \end{cor}
  
  Note that since both of $D_K$ and $D$ are essentially self-adjoint, 
  $\bbD_\epsilon$ is also an essentially self-adjoint operator on $L^2(W)^{(\rho)}$.  
  Theorem~\ref{thm:FFY=F} follows from the following estimate,  
  which is also known as the {\it coercivity} in \cite{Anghel}. 
  In fact, as in \cite{Fujita-Furuta-Yoshida2}, the $\bbZ/2$-graded Fredholm index of a coercive family of essentially self-adjoint operators  
  does not depend on a parameter of the family. 
    \begin{prop}\label{prop:coercive}
  There exist an open subset $U_\rho$ and a 
  constant $t_\rho>0$ such that $M\setminus U_{\rho}$ is compact and 
\[
\|\bbD_{\epsilon}s\|^2_{W} \geq t_\rho\kappa_\rho\|s\|^2_{W}
\] holds for any $s\in \Gamma_c(W)^{(\rho)}$ with ${\rm supp}(s)\subset U_\rho$, $\epsilon\in[0,1]$ and $t>t_\rho$, 
where $\kappa_\rho>0$ is the constant as in (\ref{cond:eigenposi}) of Definition~\ref{dfn:K-acyclic orbital system}.  
  \end{prop}
  \begin{proof}
 We take $U_\rho'$ to be the interior of $\varphi_\rho^{-1}(1)$ and  put $h:=(1-\epsilon)f^4 +\epsilon t$. 
 On $U_\rho'$ consider the square 
\begin{eqnarray*}
(D+hD_k)^2&=&D^2+(DhD_K+hD_KD)+h^2D_K^2 \\
 &=&D^2+c(dh)D_K+h(DD_K+D_KD)+h^2D_k^2. 
\end{eqnarray*}
 For any $s\in\Gamma_c(W)^{(\rho)}$ with ${\rm supp}(s)\subset U_\rho'$ we have 
 \begin{eqnarray*}
 |(c(dh)D_Ks,s)_W|&=&  |(4(1-\epsilon)f^3c(df)D_Ks,s)_W| \\ 
 &\leq&  4(1-\epsilon)\|df\|_\infty|(f^3D_Ks,s)_W| \\
 &\leq & 4 \|df\|_\infty C_\rho(hD_K^2s,s)_W 
 \end{eqnarray*}
 and 
 \begin{eqnarray*}
 |(h(DD_K+D_KD)s,s)_W|\leq C_{\rho}(hD_K^2s,s)_W. 
  \end{eqnarray*}
  It implies 
  \begin{eqnarray*}
  \|\bbD_\epsilon s\|^2_W&=&((D+hD_K)^2s,s)_W \\ 
  &\geq& ((c(dh)D_K+h(DD_K+D_KD)+h^2D_k^2)s,s)_W \\ 
  &\geq& ((-4\|df\|_\infty C_\rho-C_\rho+h)hD_K^2s,s)_W. 
  \end{eqnarray*}
  Now put $t_\rho:=4\|df\|_\infty C_\rho+C_\rho+1$ 
  and define $U_\rho$ by  
  \[
  U_{\rho}:=\{x\in U_{\rho}' \ | \  f(x)^4>t_\rho\}. 
  \] 
  Then on $U_\rho$ when $t>t_\rho$ we have 
  $
  (-4\|df\|_\infty C_\rho-C_\rho+h)h> t_\rho. 
   $
  Finally we have \footnote{This argument shows that by taking $t_\rho$ large enough and 
  $U_\rho=(f^4)^{-1}((t_\rho,\infty))$ we can refine the estimate as $\|\bbD_\epsilon s\|^2\geq \|s\|^2$ 
  for any $s\in \Gamma_c(W)^{(\rho)}$ with ${\rm supp}(s)\subset U_\rho$. } 
  \begin{eqnarray*}
  \|\bbD_\epsilon s\|^2_W
  &\geq& (t_\rho D_K^2s,s)_W \\ 
  &\geq &t_{\rho}\kappa_\rho(s, s)_W=t_{\rho}\kappa_\rho\|s\|^2_W. 
  \end{eqnarray*}
  \end{proof}
  
  Theorem~\ref{thm:FFY=F} implies that one can adopt the deformation 
  \[
  D+t\varphi_\rho^4 D_K \quad (t\gg 0)
  \]without the proper function $f$ to discuss the equivariant index $[M](\rho)=[\hat D](\rho)={\rm index}(\hat D_{\rho})$. 
  It also implies\footnote{We can apply the argument in \cite[Section~3]{Fujita-Furuta-Yoshida2} for 
$\hat D_{\rho}$ directly without using the finite propagation speed condition. 
In fact by taking a family of cut-off function $\varphi_{a, \epsilon}$ in \cite[Lemma~A.1]{Fujita-Furuta-Yoshida2} 
in a $K$-invariant way the arguments in \cite{Fujita-Furuta-Yoshida2} can still work for $\hat D_\rho$. } 
that $[M](\rho)$ satisfies the {\it excision formula}, 
{\it sum formula}, {\it invariance under continuous deformations} and {\it product formula} as stated in \cite[Section~3]{Fujita-Furuta-Yoshida2}. 
In particular if there are two data $(M, W, D, D_K, V_{\rho})$ and $(M',W', D',D'_K,V'_\rho)$ for 
the same $K$ and $\rho\in {\rm Irr}(K)$ which are isomorphic on neighborhoods of compact subsets 
$M\setminus V_\rho$ and $M'\setminus V_{\rho}'$, then the excision formula implies that the resulting indices  coincide ; 
\begin{equation}\label{eq:excision}
[M](\rho)={\rm index}(D+\varphi_\rho^4D_K)={\rm index}(D'+\varphi'^4_\rho D'_K)=[M'](\rho). 
\end{equation}
It ensures us to define the index starting from a non-complete manifold by taking an 
appropriate completion, for instance a cylindrical end as in \cite[Section~7.1]{Fujita-Furuta-Yoshida2} or \cite[Section~4.7]{LoizidesSong}.   
We gave an explanation of such a construction in Section~\ref{subsec:Non-complete case and localization formula} 
and used in Section~\ref{sec:Riemann-Roch character of proper Hamiltonian torus space}. 

\bibliography{deformdirac}

\begin{bibdiv}
\begin{biblist}

\bib{Andersen}{article}{
      author={Andersen, J.~E.},
       title={Geometric quantization of symplectic manifolds with respect to
  reducible non-negative polarizations},
        date={1997},
     journal={Comm. Math. Phys.},
      volume={183},
      number={2},
       pages={401\ndash 421},
}

\bib{Anghel}{article}{
      author={Anghel, N.},
       title={An abstract index theorem on noncompact {R}iemannian manifolds},
        date={1993},
     journal={Houston J. Math.},
      volume={19},
      number={2},
       pages={223\ndash 237},
}

\bib{Atiyah1}{book}{
      author={Atiyah, M.~F.},
       title={Elliptic operators and compact groups.},
      series={Lecture Notes in Mathematics},
   publisher={Springer-Verlag},
        date={1974},
      volume={401},
}

\bib{Atiyah-Singer}{article}{
      author={Atiyah, M.~F.},
      author={Singer, I.~M.},
       title={The index of elliptic operators. {I}},
        date={1968},
     journal={Ann. of Math. (2)},
      volume={87},
       pages={484\ndash 530},
}

\bib{Blackadar}{book}{
      author={Blackadar, B.},
       title={{$K$}-theory for operator algebras},
     edition={Second},
      series={Mathematical Sciences Research Institute Publications},
   publisher={Cambridge University Press, Cambridge},
        date={1998},
      volume={5},
}

\bib{Braverman}{article}{
      author={Braverman, M.},
       title={Index theorem for equivariant {D}irac operators on noncompact
  manifolds},
        date={2002},
     journal={$K$-Theory},
      volume={27},
      number={1},
       pages={61\ndash 101},
}

\bib{FujitaS1}{article}{
      author={Fujita, H.},
       title={{$S^1$}-equivariant local index and transverse index for
  non-compact symplectic manifolds},
        date={2016},
     journal={Math. Res. Lett.},
      volume={23},
      number={5},
       pages={1351\ndash 1367},
}

\bib{Fujitaorigami}{article}{
      author={Fujita, H.},
       title={A {D}anilov-type formula for toric origami manifolds via
  localization of index},
        date={2018},
     journal={Osaka J. Math.},
      volume={55},
      number={4},
       pages={619\ndash 645},
}

\bib{Fujita-Furuta-Yoshida1}{article}{
      author={Fujita, H.},
      author={Furuta, M.},
      author={Yoshida, T.},
       title={Torus fibrations and localization of index {I}---polarization and
  acyclic fibrations},
        date={2010},
        ISSN={1340-5705},
     journal={J. Math. Sci. Univ. Tokyo},
      volume={17},
      number={1},
       pages={1\ndash 26},
}

\bib{Fujita-Furuta-Yoshida2}{article}{
      author={Fujita, H.},
      author={Furuta, M.},
      author={Yoshida, T.},
       title={Torus fibrations and localization of index {II}: local index for
  acyclic compatible system},
        date={2014},
        ISSN={0010-3616},
     journal={Comm. Math. Phys.},
      volume={326},
      number={3},
       pages={585\ndash 633},
}

\bib{Fujita-Furuta-Yoshida3}{article}{
      author={Fujita, H.},
      author={Furuta, M.},
      author={Yoshida, T.},
       title={Torus fibrations and localization of index {III}: equivariant
  version and its applications},
        date={2014},
        ISSN={0010-3616},
     journal={Comm. Math. Phys.},
      volume={327},
      number={3},
       pages={665\ndash 689},
}

\bib{Higson-Roe}{book}{
      author={Higson, N.},
      author={Roe, J.},
       title={Analytic {$K$}-homology},
      series={Oxford Mathematical Monographs},
   publisher={Oxford University Press, Oxford},
        date={2000},
        note={Oxford Science Publications},
}

\bib{Hochs-SongI}{article}{
      author={Hochs, P.},
      author={Song, Y.},
       title={An equivariant index for proper actions {I}},
        date={2017},
     journal={J. Funct. Anal.},
      volume={272},
      number={2},
       pages={661\ndash 704},
}

\bib{Kasparov}{article}{
      author={Kasparov, G.},
       title={Elliptic and transversally elliptic index theory from the
  viewpoint of {$KK$}-theory},
        date={2016},
     journal={J. Noncommut. Geom.},
      volume={10},
      number={4},
       pages={1303\ndash 1378},
}

\bib{Kubota}{article}{
      author={Kubota, Y.},
       title={The joint spectral flow and localization of the indices of
  elliptic operators},
        date={2016},
     journal={Ann. K-Theory},
      volume={1},
      number={1},
       pages={43\ndash 83},
}

\bib{Loizides-Song-Meinrenken}{article}{
      author={Loizides, Y.},
      author={Meinrenken, E.},
      author={Song, Y.},
       title={{Spinor modules for Hamiltonian loop group spaces}},
        date={2020},
        ISSN={1527-5256; 1540-2347/e},
     journal={{J. Symplectic Geom.}},
      volume={18},
      number={3},
       pages={889\ndash 937},
}

\bib{LoizidesSong}{article}{
      author={Loizides, Y.},
      author={Song, Y.},
       title={Quantization of {H}amiltonian loop group spaces},
        date={2019},
     journal={Math. Ann.},
      volume={374},
      number={1-2},
       pages={681\ndash 722},
}

\bib{Loizides-Song2}{article}{
      author={Loizides, Y.},
      author={Song, Y.},
       title={Norm-square localization and the quantization of hamiltonian loop
  group spaces},
        date={2020},
     journal={J. Funct. Anal.},
      volume={278},
      number={9},
       pages={45 p},
}

\bib{LoizidesRodsphonSong}{article}{
      author={Loizides, Yiannis},
      author={Rodsphon, Rudy},
      author={Song, Yanli},
       title={A {KK}-theoretic perspective on deformed {D}irac operators},
        date={2021},
     journal={Adv. Math.},
      volume={380},
}

\bib{Ma-Zhangarxiv}{article}{
      author={Ma, X.},
      author={Zhang, W.},
       title={Geometric quantization for proper moment maps: the {V}ergne
  conjecture},
        date={2014},
        ISSN={0001-5962},
     journal={Acta Math.},
      volume={212},
      number={1},
       pages={11\ndash 57},
}

\bib{Nicolaescu}{article}{
      author={Nicolaescu, L.~I.},
       title={On the space of {F}redholm operators},
        date={2007},
     journal={An. \c{S}tiin\c{t}. Univ. Al. I. Cuza Ia\c{s}i. Mat. (N.S.)},
      volume={53},
      number={2},
       pages={209\ndash 227},
}

\bib{ParadanlocRR}{article}{
      author={Paradan, P.-E.},
       title={Localization of the {R}iemann-{R}och character},
        date={2001},
     journal={J. Funct. Anal.},
      volume={187},
      number={2},
       pages={442\ndash 509},
}

\bib{Takata3}{unpublished}{
      author={Takata, D.},
       title={An infinite-dimensional index theorem and the
  {H}igson-{K}asparov-{T}rout algebra},
        note={arXiv:1811.06811},
}

\bib{Takata}{article}{
      author={Takata, D.},
       title={An analytic {$LT$}-equivariant index and noncommutative
  geometry},
        date={2019},
     journal={J. Noncommut. Geom.},
      volume={13},
      number={2},
       pages={553\ndash 586},
}

\bib{Takata2}{article}{
      author={Takata, D.},
       title={L{T}-equivariant index from the viewpoint of {KK}-theory. {A}
  global analysis on the infinite-dimensional {H}eisenberg group},
        date={2020},
     journal={J. Geom. Phys.},
      volume={150},
}

\bib{Tian-Zhang}{article}{
      author={Tian, Y.},
      author={Zhang, W.},
       title={An analytic proof of the geometric quantization conjecture of
  {G}uillemin-{S}ternberg},
        date={1998},
     journal={Invent. Math.},
      volume={132},
      number={2},
       pages={229\ndash 259},
}

\bib{Witten}{article}{
      author={Witten, E.},
       title={Supersymmetry and {M}orse theory},
        date={1982},
     journal={J. Differential Geometry},
      volume={17},
      number={4},
       pages={661\ndash 692 (1983)},
}

\bib{Yoshida100}{unpublished}{
      author={Yoshida, T.},
       title={Adiabatic limits, theta functions, and geometric quantization},
        note={ar{X}iv:1904.04076},
}

\end{biblist}
\end{bibdiv}

\end{document}